\documentclass[a4paper,11pt]{article}

\usepackage{bm}
\usepackage{amsmath}
\usepackage{amsthm}
\usepackage{wasysym}
\usepackage{amssymb}
\usepackage{amsfonts}
\usepackage{graphicx}
\usepackage[english]{babel}
\usepackage[utf8]{inputenc}%a la place d'applemac utf8
\usepackage[T1]{fontenc}
\usepackage{mathrsfs}
\usepackage{enumerate}
\usepackage{version}
\usepackage{calc}
\usepackage{subfigure}
\usepackage{color}

 \usepackage{pstricks,pst-math,pst-xkey}
 \usepackage{float}
 \usepackage{indentfirst}%%%%%%%%%%%%%%%%%%%%%%%%%%
\newtheorem{thm}{Theorem}
\newtheorem{cor}{Corollary}
\newtheorem{lem}{Lemma}

\newtheorem{definition}{Definition}
\newtheorem{prop}{Proposition}

\theoremstyle{remark}

\newcommand{\nc}{\newcommand}

\nc{\Cal}[1]{{\mathcal {#1}}}
\nc{\Bf}[1]{{\bf {#1}}}
\nc{\Em}[1]{{\em {#1}}}
\nc{\Rm}[1]{{\rm {#1}}}
\nc{\Type}[1]{{\tt {#1}}}
\nc{\nonu}{\nonumber}
\nc{\Headline}[1]{\noindent{\bf {#1}: }}

\nc{\EqnRef}[1]{(\ref{#1})}
\nc{\DefRef}[1]{\Bf{Definition \ref{#1}}}
\nc{\LemRef}[1]{\Bf{Lemma \ref{#1}}}
\nc{\ProRef}[1]{\Bf{Proposition \ref{#1}}}
\nc{\ThmRef}[1]{\Bf{Theorem \ref{#1}}}
\nc{\CorRef}[1]{\Bf{Corollary \ref{#1}}}
\nc{\SecRef}[1]{\Bf{Section \ref{#1}}}
\nc{\RemRef}[1]{\Bf{Remark \ref{#1}}}
\nc{\ChpRef}[1]{\Bf{Chapter \ref{#1}}}

\nc{\defref}[1]{Definition \ref{#1}}
\nc{\lemref}[1]{Lemma \ref{#1}}
\nc{\proref}[1]{Proposition \ref{#1}}
\nc{\thmref}[1]{Theorem \ref{#1}}
\nc{\corref}[1]{Corollary \ref{#1}}
\nc{\secref}[1]{Section \ref{#1}}
\nc{\remref}[1]{Remark \ref{#1}}
\nc{\chpRef}[1]{Chapter \ref{#1}}

\nc{\Proof}{\begin{proof}}
\nc{\EndProof}{\end{proof}}

\nc{\apriori}{{\it a priori }}
\nc{\Apriori}{{\it A priori }}
\nc{\Holder}{H\"{o}lder }

\nc{\Sty}{\displaystyle}
\nc{\MathSty}[1]{$\Sty{#1}$}
\def\Beqn#1\Eeqn{\begin{equation}#1\end{equation}}

\def\upchi{\raise2pt\hbox{$\chi$}}
\def\upnu{\raise0pt\hbox{$\nu$}}
\def\upstroke{\raise3pt\hbox{$|$}}
\nc{\Bdry}{\partial}
\nc{\Abs}[1]{\left|{#1}\right|}
\nc{\Ave}[1]{\bar{#1}}
\nc{\CurBrac}[1]{\left\{{#1}\right\}}
\nc{\curBrac}[1]{\{{#1}\}}
\nc{\Brac}[1]{\left({#1}\right)}
\nc{\brac}[1]{({#1})}
\nc{\SqrBrac}[1]{\left[{#1}\right]}
\nc{\sqrbrac}[1]{[{#1}]}
\nc{\converge}{\rightarrow}
\nc{\Converge}{\longrightarrow}
\nc{\WeakConverge}{\rightharpoonup}
\nc{\Del}{\triangle}
\nc{\Div}{\mbox{\rm{div}}}
\nc{\Equivalent}{\Longleftrightarrow}
\nc{\IntOverR}{\int_{-\infty}^{\infty}}
\nc{\IntOverRp}{\int_{0}^{\infty}}
\nc{\IntOverRm}{\int_{-\infty}^{0}}
\nc{\Imply}{\Longrightarrow}
\nc{\InnProd}[2]{\left\langle{#1},\,{#2}\right\rangle}
\nc{\innprod}[2]{\langle{#1},\,{#2}\rangle}
\nc{\BracInnProd}[2]{\left({#1},\,{#2}\right)}
\nc{\bracInnprod}[2]{({#1},\,{#2})}
\nc{\SqrinnProd}[2]{\left[{#1},\,{#2}\right]}
\nc{\sqrinnprod}[2]{[{#1},\,{#2}]}
\nc{\Lap}{\triangle}
\nc{\Lip}{\mbox{\rm{Lip}}}
\nc{\Lover}[1]{\frac{1}{#1}}
\nc{\Grad}{\nabla}
\nc{\MapTo}{\longrightarrow}
\nc{\Min}{\wedge}
\nc{\Max}{\vee}
\nc{\Norm}[1]{\left\|#1\right\|}
\nc{\norm}[1]{\|#1\|}
\nc{\SingleNorm}[1]{\left|#1\right|}
\nc{\singlenorm}[1]{|#1|}
\nc{\Pair}[2]{\left\langle{#1},\,{#2}\right\rangle}
\nc{\pair}[2]{\langle{#1},\,{#2}\rangle}
\nc{\DD}[1]{\frac{d}{d{#1}}}
\nc{\PDD}[1]{\frac{\partial}{\partial{#1}}}
\nc{\Spt}{\mbox{\rm{spt}}}
\nc{\spt}{\mbox{\rm{spt}}}
\nc{\SuchThat}{\ni}
\nc{\Sum}[2]{\sum_{#1}^{#2}}
\nc{\wtilde}[1]{\widetilde{#1}}
\nc{\Text}[1]{\mbox{#1}}
\nc{\TextMath}[1]{\mbox{\,\,\,{#1}\,\,\,}}
\nc{\intersect}{\cap}
\nc{\Intersect}{\bigcap}
\nc{\union}{\cup}
\nc{\Union}{\bigcup}
\nc{\ColTwo}[2]
{\left(\begin{array}{c}{#1}\\{#2}\end{array}\right)}
\nc{\ColThree}[3]
{\left(\begin{array}{c}{#1}\\{#2}\\{#3}\end{array}\right)}
\nc{\ColFour}[4]
{\left(\begin{array}{c}{#1}\\{#2}\\{#3}\\{#4}\end{array}\right)}
\nc{\ColFive}[5]
{\left(\begin{array}{c}{#1}\\{#2}\\{#3}\\{#4}\\{#5}\end{array}
\right)}
\nc{\coltwo}[2]
{\begin{array}{c}{#1}\\{#2}\end{array}}
\nc{\colthree}[3]
{\begin{array}{c}{#1}\\{#2}\\{#3}\end{array}}
\nc{\colfour}[4]
{\begin{array}{c}{#1}\\{#2}\\{#3}\\{#4}\end{array}}
\nc{\colfive}[5]
{\begin{array}{c}{#1}\\{#2}\\{#3}\\{#4}\\{#5}\end{array}}

\nc{\Vol}[2]{\Bf{#1}({#2})}

\nc{\colorblue}[1]{\textcolor{blue}{#1}}
\nc{\colorBLUE}[1]{\textcolor{blue}{{\bf #1}}}
\nc{\colorred}[1]{\textcolor{red}{#1}}
\nc{\colorRED}[1]{\textcolor{red}{{\bf #1}}}
\def\a{\alpha}

\def\dl{\delta}

\def\d{\partial}

\def\t{\tau}

\def\R{\mathbb{R}}
\def\N{\mathbb{N}}

\def\Div{\text{\rm  div}}

\def\f{\varphi}
\def\Z{\mathbb{Z}}

\def\N{\mathbb{N}}

\nc{\levelsets}{1}
\nc{\setgraph}{2}
\nc{\griddisect}{3}
\author{Oleksandr~\textsc{Misiats}\\
 Department of Mathematics, Purdue University\\
West Lafayette, IN, 47907, USA
\\{\tt omisiats@purdue.edu}
\and Nung Kwan~\textsc{Yip}\\
 Department of Mathematics, Purdue University\\
West Lafayette, IN, 47907, USA
\\{\tt yip@math.purdue.edu}
}

\title{Convergence of Space-Time Discrete Threshold Dynamics to
Anisotropic Motion by Mean Curvature}

\begin{document}
\maketitle
\begin{abstract}
We analyze the continuum limit of a thresholding algorithm for motion by mean
curvature of one dimensional interfaces in various space-time discrete
regimes. The algorithm can be viewed as a time-splitting scheme for the
Allen-Cahn equation which is a typical model for the motion of materials phase
boundaries. Our results extend the existing statements which are applicable
mostly in semi-discrete (continuous in space and discrete in time) settings.
The motivations of this work are twofolds: to investigate the interaction
between multiple small parameters in nonlinear singularly perturbed problems,
and to understand the anisotropy in curvature for interfaces in spatially
discrete environments. In the current work, the small parameters are the
the spatial and temporal discretization step sizes $\triangle x = h$ and
$\triangle t = \t$.
We have identified the limiting description of the interfacial velocity
in the (i) sub-critical ($h \ll \t$),
(ii) critical ($h  = O(\t)$), and
(iii) super-critical ($h \gg \t$) regimes.
The first case gives the classical isotropic motion by mean curvature,
while the second produces intricate pinning and de-pinning phenomena
and anisotropy in the velocity function of the interface. The last case
produces no motion (complete pinning).
\end{abstract}
\section{Introduction and Main Results}
The current paper addresses convergence issues related to
a thresholding scheme for motion by mean curvature. The key is the
analysis of the algorithm in the \Em{space-time discrete} setting in
which there are \Em{two small parameters} - the step sizes in the spatial
and temporal directions. The ultimate results depend on
the relative sizes of these parameters.

The analysis of motion by mean curvature (in which the normal velocity of
a moving manifold is given by its mean curvature) is an active area.
Not only it is interesting in geometry in its own right, it also finds
many applications in materials science and image processing. It is a
prototype of a gradient flow with respect to the area functional.
Due to the possibility of singularity formation and topological changes
of the evolving surface, elaborate approaches need to be used. These include
(i) varifold formulation,
(ii) the viscosity solutions, and
(iii) singularly perturbed reaction diffusion equations.

The thresholding scheme is a particularly simple algorithm to capture the
key feature of (iii). It is essentially a time splitting scheme. The first
step is diffusion while the second step is thresholding to mimick the
fast reaction due to the nonlinear term. This is heuristically proposed in
\cite{M} and rigorously proved in \cite{Evans, Barles} in the continuous
space and discrete time setting. See also the work
\cite{ChenMaxPrin} for an analysis of an reaction diffusion equation
in which both space and time variables are discrete.
However, so far all the rigorous results essentially works in the case
when the interfacial structure is well-resolved. We call this the
``sub-critical'' regime. When this is not the case, intricate
pinning and depinning of the interface can happen.
This is analogous to a gradient flow in a \Em{highly wiggling} or
\Em{oscillatory energy landscape}.
The motion also demonstrates anisotropy of the normal velocity.
The motivation of the current paper is to capture these phenomena
quantitatively and relate them to the underlying small parameters
in the algorithm.

The most relevant reaction diffusion equation for motion by mean curvature
is the following Allen-Cahn equation:
\Beqn\label{ACeqn}
\frac{\partial u}{\partial t} = \Lap u - \Lover{\epsilon^2}W'(u)
\Eeqn
In the above $W$ is the double well potential $W(u) = (1-u^2)^2$ and
$\epsilon$ is a small parameter.
The qualitative behavior of the solution is that the
underlying ambient space is quickly partitioned into two domains on which
$u$ takes on the values $1$ and $-1$ which are the minima of $W$.
The function $u$ also makes a smooth but rapid transition with thickness
$O(\epsilon)$ between the two domains. The key is then
to understand the dynamics of this transition layer, in the limit of
$\epsilon\Converge 0$. It is proved in various settings that the limiting
motion is motion by mean curvature \cite{KohnBronsard, deMottoniSchatzman,
ChenMaxPrin, IlmanenAC, ESS}.

%It seems difficult to fully analyze singularly perturbed reaction diffusion
%equation in the (pratically realistic) ``critical regime''.
%\begin{itemize}
%\item	\Bf{Survey of numerical papers and results};
%\item	\Bf{The result of Chen, Elliott...};
%\item	\Bf{Pinning of the solution if $\Delta \sim \epsilon$}
%\item	\Bf{Motivation of the current paper}
%\end{itemize}

As the thresholding scheme is very simple to implement and describe,
we embark on its analysis demonstrating the interplay between two
small parameters. The scheme is a time splitting approach
to solve \eqref{ACeqn} (in the regime $\epsilon\ll 1$).
Given an initial shape $\Omega_0$,
and its boundary (or often called the interface) $\Gamma_0 = \Bdry\Omega_0$,
a sequence of functions
$\CurBrac{u_{k}}_{k\geq 0}$ is constructed in the following manner:
for $k=0$, we define
\[
u_0(x) = \mathbf{1}_{\Omega_0}(x) - \mathbf{1}_{\Omega_0^c}
= \left\{
\begin{array}{lll}
1 & \text{for} & x\in\Omega_0,\\
-1 & \text{for} & x\in\Omega_0^c,
\end{array}
\right.
\]
then the following two steps are alternately performed
(for $k = 0, 1, 2, \ldots$),
\begin{itemize}
\item	\Bf{diffusion step:}
\Beqn\label{threshold1}
\begin{array}{ll}
\frac{\partial v}{\partial t} = \Lap v & \text{for}\,\,\,\,\,\,
0 < t < \t;\\
v(x, 0) = u_{k}(x)
\end{array}
\Eeqn

\item	\Bf{thresholding step:}
\Beqn\label{threshold2}
u_{k+1}(x) = \text{sign}(v(x,\t))
\Eeqn
\end{itemize}
Note that the second step above is to mimick the fast reaction term
which drives $u$ to $1$ or $-1$, the minima of $W$. The solution of the
problem is captured by the sequence of subsets where $u_k$ attains the value
$1$, $\CurBrac{x: u_k(x) = 1}$. Precisely,
we define the time dependent set and interface as
\[
\Omega^{\t}(t) = \CurBrac{x: u_k(x) \geq 0,\,\,\,\text{for}\,\,\,
k\t \leq t < (k+1)\t}\,\,\,\text{and}\,\,\,
\Gamma^{\t}(t) = \Bdry\Omega^{\t}(t).
\]
Then as $\t\Converge 0$, $\Omega^{\t}(t)$
(or $\Gamma^{\t}(t)$) has been shown to converge to motion by
mean curvature in the viscosity setting \cite{Evans, Barles}.

%The heuristic reasoning why the Allen-Cahn equation and the thresholding
%scheme captures motion by mean curvature is given in the appendix for
%self-containedness of presentation.

%\subsection{Notation.}
Now we describe some notations and the algorithm for the space-time discrete
version of the above thresholding scheme \eqref{threshold1} and
\eqref{threshold2}.
Let $\Omega \subset \R^2$ be a bounded, smooth domain, and
$\Gamma = \d\Omega$ be its boundary. Let $h>0$ be the spatial discretization
step size. Define
\Beqn
\Omega^h :=\{ (m, n) \in \Z^2: {\rm dist}[(nh, mh), \Omega] \leq h\}
\Eeqn
which are the indices of the lattice points inside $\Omega$.
Let again $\t>0$ be the size of the time step. Given an initial set
$\Omega_0$ and its discrete version $\Omega_0^h$, the discrete
thresholding scheme produces
$\CurBrac{u^{m,n}_{k}}_{k\geq 0, (m,n) \in \Z^2}$ as follows.
Let
\begin{equation}\label{scheme-11}
u^{m,n}_0 =
\mathbf{1}_{\Omega_0^h}(m,n) -
\mathbf{1}_{(\Omega_0^h)^c}(m,n)
= \left\{
\begin{array}{lll}
1 & \text{for} & (m,n)\in\Omega_0^h,\\
-1 & \text{for} & (m,n)\in(\Omega_0^h)^c,
\end{array}
\right.
\end{equation}
For simplicity, we will use $u$ to denote the discrete function
$\CurBrac{u^{m,n}: (m,n)\in\Z^2}$.
Then similar to the continuous space case,
the following two steps are alternately performed
(for $k=0, 1, 2, \ldots$):
\begin{itemize}
\item	\Bf{solution of semi-discrete heat equation}:
for $(m, n) \in \Z^2$ and $0 < t \leq \t$,
\begin{equation}\label{semidiscrete}
\begin{cases}
\frac{d}{dt}w^{m,n} (t) = \frac{1}{h^2}
[w^{m+1, n}(t) + w^{m-1,n}(t) + w^{m, n+1}(t) + w^{m,n-1}(t) - 4 w^{m,n}(t)], \\
w^{m,n}(0)  = u^{m,n}_k
\end{cases}
\end{equation}

\item	\Bf{thresholding step}:
for $n, m \in \Z$,
\begin{equation}\label{defn_of_uk}
u_{k+1}^{m,n} := {\rm sign} [w^{m,n}(\t)]
= {\rm sign} \left(\Brac{S^h(\t)[u_k]}^{m,n}\right),
\end{equation}
where the $S^h(\cdot)$ is the solution operator of \eqref{semidiscrete},
i.e. $w(t) = S^h(t)[u_k]$.
\end{itemize}
%%where
%%$${\hat \Omega}_k := \{(m, n) | w(m,n; (k+1) h) > \frac{1}{2}\}.$$
%The motion introduced above may be alternatively described as follows:
%\[
%u_{k+1}^{m,n} := {\rm sign} \left(S^h(t)[u_k^{m,n}]\right),
%\]
%where $S^h(t)$ is a semigroup associated with the equation
%\eqref{semidiscrete}, i.e.
%$S^h(\tau)[u_k^{m,n}] = w^{m,n} (\t)$.

%It was established in \cite{Evans} and \cite{Barles} that the level curves of a continuum heat equation move by mean curvature. Our goal is to study the law of motion of the level curves of the semi-discrete equation (\ref{d_heat}).

%We start with the well-posedness result for (\ref{semidiscrete}).
%
%\begin{lem}\label{L1.wellposed}
%Given the initial data $u_k^{m,n}$, the problem (\ref{semidiscrete}) has a unique solution for all $h>0$ and $\tau>0$
%\end{lem}

Note that the above space time discrete scheme involves two small
parameter $\t > 0$ and $h > 0$. Assume that $\t$ and $h$
are related through
\begin{equation}\label{relation}
h = C (\t)^\gamma, \text{ where } \gamma > 0 \text{ and } C > 0
\end{equation}
Three major cases are possible:
\begin{description}
\item[\hspace{80pt}Case 1.] $\gamma > 1$, i.e. $h \ll \t$,
called ``subcritical''.

\item[\hspace{80pt}Case 2.] $\gamma = 1$, i.e. $\t = \mu h$,
where $\mu = const$, called ``critical''.
\item[\hspace{80pt}Case 3.] $\gamma < 1$, i.e. $h \gg \t$,
called ``supercritical''.
\end{description}
\noindent
Roughly speaking, the main result in Case 1 is that the level curves of a
discrete heat equation move according to the motion by mean curvature.
This gives the same result as \cite{Barles}.
Case 2 gives a version of anisotropic curvature dependent motion
which demonstrates pinning of the interface when the curvature of the
interface is too small. In Case 3, there is no motion at all.

\subsection{Curvature Dependent Motion and Viscosity Solutions}
As mentioned earlier, singularities and topological changes can occur for
motion by mean curvature. Different mathematical approaches are invented
to define the solution for all time. Due to the presence of maximum
principle, we find the viscosity solution to be
the most suitable and convenient for our problem. We spend a moment to
briefly describe this method which can produce a
``unique'' global in time solution.

The essential idea is to represent
the moving interface $\Gamma_t$ as the zero level set of
a function $u(x,t)$:
\Beqn\label{FrontZeroSet}
\Gamma_t = \CurBrac{x: u(x,t) = 0}
\Eeqn
The function $u$ is thus often called the \Em{level set function}. It solves
an appropriate partial differential equations related to the motion law of
$\Gamma_t$.
The main result in this approach is that in the space of uniformly continuous
functions, there is a unique solution $u$ and the set $\Gamma_t$ does not
depend on the initial data $u(\cdot, 0)$ as long as it correctly captures
the interior and exterior domains of $\Gamma_0$.
On the other hand, this set-up does not a priori ensure that $\Gamma_t$
corresponds to a manifold in any geometric sense. This can happen
if $\Gamma_t$ has positive $n$-dimensional Lebesgue measure in which
case $\Gamma_t$ is said to
\Em{fatten} or develop \Em{non-empty interior}.
It also means that the solution of the geometric evolution can be
non-unique
as $\Bdry\CurBrac{x: u(x,t) > 0} \neq \Bdry\CurBrac{x: u(x,t) < 0}$.
Conditions preventing this from happening
are discussed in \cite{SIAMControl}. On the other hand,
a definition of \Em{generalized front} is used so that
a ``unique solution'' can be defined. This approach defines the interface
as the following triplet of objects:
\Beqn\label{GenFrontDef}
\Gamma_t = \CurBrac{x: u(x,t)=0},\,\,\,\,\,\,
D^+_t = \CurBrac{x: u(x,t)>0},\,\,\,\,\,\,
D^-_t = \CurBrac{x: u(x,t)<0}.
\Eeqn
\begin{center}
\includegraphics [width = 1.8in, angle = 270]{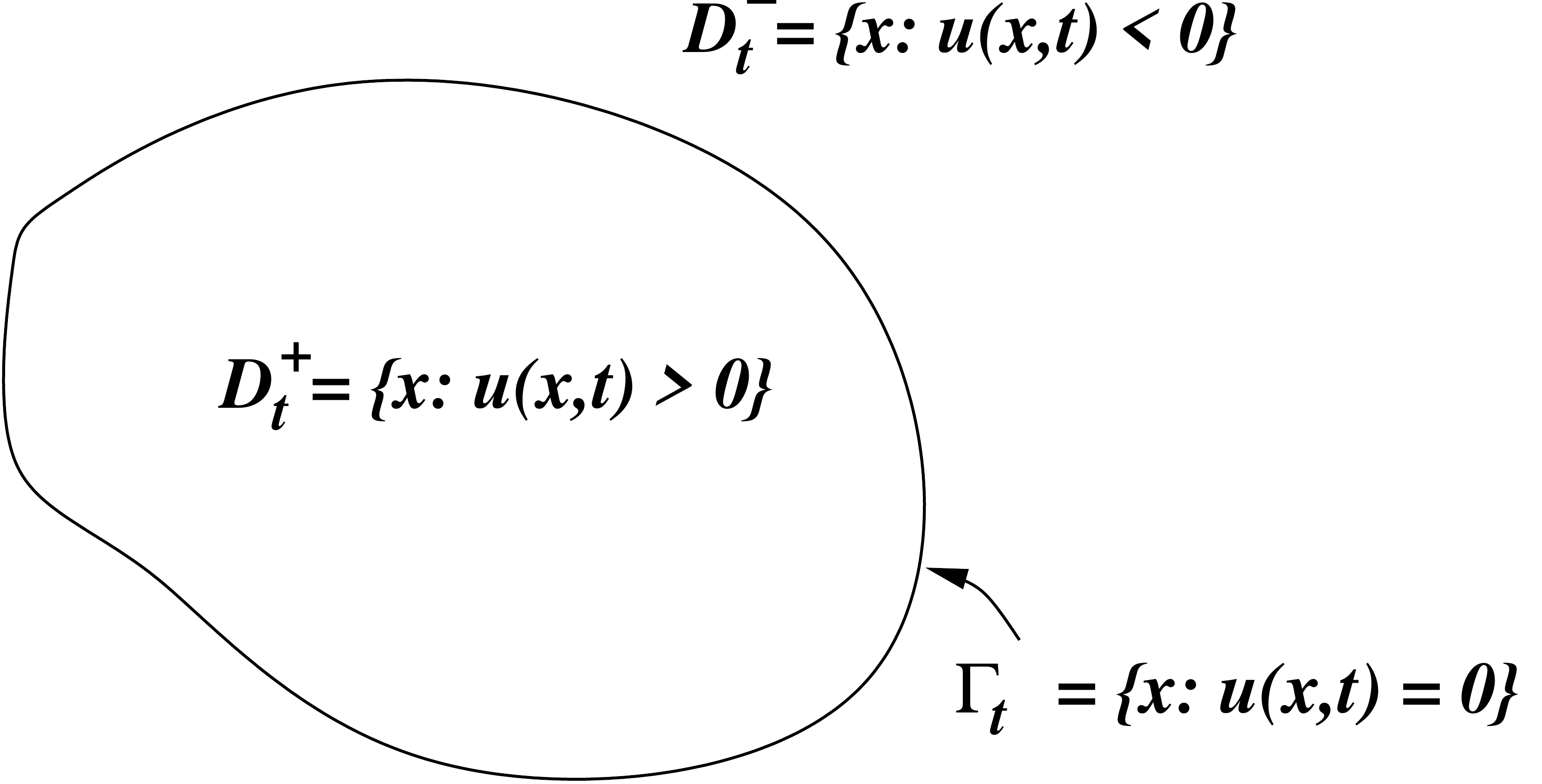}\\
Figure \levelsets.
\end{center}
(See Fig. \levelsets.)
Then $D^+_t$ ($D^-_t$) is called the \Em{interior (exterior)} of the front
$\Gamma_t$.
It is shown that the map:
\Beqn\label{GenFrontMap}
E_t: (\Gamma_0, D_0^+, D_0^-) \MapTo (\Gamma_t, D_t^+, D_t^-)
\Eeqn
is well defined. We refer to \cite{SouganidisBarlesARMA}
%SouganidICM
for a more detailed description.

Next we describe the equation for general curvature dependent front
propagation. We follow the exposition in \cite{Pires}.
Given an interface (hypersurface) $\Gamma$ in $\R^n$.
We consider its motion described by a normal velocity function $V$
of the following form:
\Beqn
V = V(D\nu, \nu)
\Eeqn
where $\nu$ is the unit (outward) normal of $\Gamma$.
(Note that $D\nu$ is a symmetric matrix.) The above motion law is sometimes
called \Em{anisotropic curvature motion}.
If $V= \text{div}(Dn)$, then the motion is called \Em{(isotropic)
motion by mean curvature}. If we want to represent the moving interface by
\eqref{FrontZeroSet} or \eqref{GenFrontDef}, then the function $u$ needs
to solve the following partial differential equation:
\Beqn\label{GenFrontEqn}
u_t + F(D^2u, Du) = 0
\Eeqn
where the function $F$ is related to $V$ in the following way:
\Beqn
F(X,p) = -\Abs{p}V\Brac{
-\Abs{p}^{-1}
\Brac{I - \bar{p}\times\bar{p}}X\Brac{I - \bar{p}\times\bar{p}},
-\bar{p}
}
\Eeqn
where $\bar{q}=\Abs{p}^{-1}p$.
In order for the viscosity solution approach to work, the following
monotonicity condition for $V$ is crucial:
\Beqn
\text{$V$ is nondecreasing, i.e.
for all $X\leq Y$ and $p$, then $V(X,p)\geq V(Y,p)$.
}
\Eeqn
The above property is translated to the function $F$ as
\Beqn
\text{for all $X\leq Y$ and $p$, then $F(X,p)\geq F(Y,p)$.
}
\Eeqn

With the above set-up, then it can be shown that equation
\eqref{GenFrontEqn} is well-posed in the space of
\Em{uniformly continuous functions}.
Given an initial manifold $\Gamma_0 = \Bdry\Omega_0$,
a usual choice of the initial data $u(x,0) = u_0(x)$ for \EqnRef{GenFrontEqn}
is given by the sign distance function to $\Gamma_0$:
\Beqn\label{InitSignDist}
d_0(x) = \text{sdist}(x, \Gamma_0)
= \left\{
\begin{array}{ll}
\text{dist}(x,\Gamma_0), & x\in\Omega_0,\\
-\text{dist}(x,\Gamma_0), & x\in\Omega_0^c,
\end{array}
\right.
\Eeqn
The map $E_t$ in \eqref{GenFrontMap} is independent of any
uniformly continuous initial data as long as it has the same sign
as $d_0$: $u_0 > 0$ on $\Omega_0$ and $u_0 < 0$ on $\Omega_0^c$.

The definition of viscosity solution of \eqref{GenFrontEqn} is given in
Section \ref{ViscDefConv} where a general approach to prove convergence
of various approximating schemes is also given.
Of the \Em{stability} and \Em{consistency} conditions generally
required for most convergence proof, for the current algorithm, the former
is quite easy to satisfy by means of maximum principle. The crux
of the matter is the latter condition which is the key result of our paper
for the case of space time discrete thresholding scheme. Once we have this,
then we can more or less quote the general convergence result.

\subsection{Main result: Sub-Critical Case}\label{ViscDefConv}
Our most complete result is the sub-critical case for which we can prove
convergence to motion by mean curvature in the viscosity sense.
In this case, the velocity function is given by $V(Dn,n) =\text{Div}(Dn)$.
Hence \EqnRef{GenFrontEqn} becomes
\begin{equation}\label{meancurv}
\frac{\partial u}{\partial t}
= |\nabla u|\text{div}\Brac{\frac{\nabla u}{|\nabla u|}}
= \Delta u - \frac{(D^2u Du| Du)}{|Du|^2}
\end{equation}
For convenience, we give the definition of viscosity solution specifically
for this case.

\begin{definition} A locally bounded upper semicontinuous (usc) function
(respectively, lower semicontinuous (lsc)) function $u$
is a viscosity subsolution (respectively, supersolution) of
\eqref{meancurv}, if for all $\phi\in C^2(\R^N\times(0,+\infty))$,
and if $(x,t)\in\R^N\times(0,+\infty)$
is a local maximum point of $u-\phi$, then one has
\Beqn\label{SubNonDeg}
\frac{\partial \phi}{\partial t}(x,t)
-\Brac{\Lap\phi- \frac{\Brac{D^2\phi D\phi | D\phi}}{\Abs{D\phi}^2}
}(x,t)\leq 0,\,\,\,\,\,\,\text{if $D\phi(x,t)\neq 0$,}
\Eeqn
and
\Beqn\label{SubDeg}
\frac{\partial \phi}{\partial t}(x,t)
-\Lap\phi(x,t) + \lambda_{\text{min}}(D^2\phi(x,t))\leq 0,
\,\,\,\,\,\,\text{if $D\phi(x,t)=0$,}
\Eeqn
where $\lambda_{\text{min}}(D^2\phi(x,t))$ is the least eigenvalue
of $D^2\phi(x,t)$.
(Respectively, if for all $\phi\in C^2(\R^N\times(0,+\infty))$,
and if $(x,t)\in\R^N\times(0,+\infty)$
is a local minimum point of $u-\phi$, then one has
\Beqn\label{SuperNonDeg}
\frac{\partial \phi}{\partial t}(x,t)
-\Brac{\Lap\phi- \frac{\Brac{D^2\phi D\phi | D\phi}}{\Abs{D\phi}^2}
}(x,t)\geq 0,\,\,\,\,\,\,\text{if $D\phi(x,t)\neq 0$,}
\Eeqn
and
\Beqn\label{SuperDeg}
\frac{\partial \phi}{\partial t}(x,t)
-\Lap\phi(x,t) + \lambda_{\text{max}}(D^2\phi(x,t))\geq 0,
\,\,\,\,\,\,\text{if $D\phi(x,t)=0$,}
\Eeqn
where $\lambda_{\text{max}}(D^2\phi(x,t))$ is the maximum (or principle)
eigenvalue of $D^2\phi(x,t)$.)
\end{definition}
On the other hand, a simpler characterization can be given.
\begin{prop}\label{comp}\cite[Prop. 2.2]{Barles}
A locally bounded upper semicontinuous (usc) function $u$ is a
viscosity subsolution (respectively supersolution) of (\ref{meancurv})
iff if satisfies \eqref{SubNonDeg} and
\begin{equation}\label{consistency2sub}
\frac{\d \varphi}{\d t}(x,t) \leq 0 \ \text{ if } D \varphi(x,t) = 0 \ \text{ and } D^2 \varphi(x,t) = 0,
\tag{\ref{SubDeg}'}
\end{equation}
respectively, \eqref{SuperNonDeg} and
\begin{equation}\label{consistency2super}
\frac{\d \varphi}{\d t}(x,t) \geq 0 \ \text{ if } D \varphi(x,t) = 0 \ \text{ and } D^2 \varphi(x,t) = 0,
\tag{\ref{SuperDeg}'}
\end{equation}
\end{prop}
%\begin{remark}
%The Definition \ref{comp} requires considering the test functions with compact support, which is different from the definition of the viscosity solution in \cite{barles}.
%\end{remark}

The consistency proof of the thresholding scheme relies on the following
result
\begin{prop}\cite[Prop 4.1]{Barles}
If $\Brac{\phi_h}_{h}$ is a sequence of smooth functions bounded in $C^{2,1}$
and converging locally in $C^{2,1}$ to a function $\phi$ and $(x_h, t_h)$ is
a sequence of points converging to $(x,t)\in\R^N\times(0,\infty)$ such that
$\phi_h(x_h,t_h)=0$, then if $D\phi(x,t)\neq 0$,
\begin{multline}
\liminf_h\Lover{h^{\Lover{2}}}\Brac{\Lover{2}
-\Lover{(4\pi h)^{\frac{N}{2}}}\int_{\CurBrac{\phi_h(\cdot, t_h - h)\geq 0}}
\exp\Brac{-\frac{\Abs{x_h-y}^2}{4h}}\,dy}\\
\geq
\Lover{2\sqrt{\pi}\Abs{D\phi(x,t)}}\Brac{
\frac{\partial\phi}{\partial t} - \Lap\phi +
\frac{\Brac{D^2\phi D\phi | D\phi}}{\Abs{D\phi}^2}
}(x,t).
\end{multline}
Moreoever, if $D\phi(x,t)=0$ and $D^2\phi(x,t)=0$ and if the inequality
\Beqn
\Lover{2}
-\Lover{(4\pi h)^{\frac{N}{2}}}\int_{\CurBrac{\phi_h(\cdot, t_h - h)\geq 0}}
\exp\Brac{-\frac{\Abs{x_h-y}^2}{4h}}\,dy\leq 0
\Eeqn
holds for a sequence of $h$ converging to $0$, then
\Beqn
\frac{\partial\phi}{\partial t}(x,t)\leq 0
\Eeqn
\end{prop}

With the above preparation, we are ready to present our result.
We start by introducing
\begin{equation}\label{subsol}
\underline{u}(x,t) = \liminf_{k\t \to t; \\
(m h, n h) \to x} u_k^{m,n}
\end{equation}
and
\begin{equation}\label{supsol}
\bar{u}(x,t) = \limsup_{k \t \to t; \\
(m h, n h) \to x} u_k^{m,n}
\end{equation}
\begin{thm}[Sub-critical case]\label{P1.Case1}
Assume $\gamma > 1$. Then the functions $\bar{u}(x,t)$ and  $\underline{u}(x,t)$ are viscosity subsolutions and supersolutions of (\ref{meancurv}), respectively.
\end{thm}
A few remarks about the consequence of the above statement.
\begin{enumerate}
\item
Let $u$ be the solution of \eqref{meancurv} with the initial data $u_0$ given
for example by \eqref{InitSignDist}. Then the sets $\Gamma_t$, $D_t^+$,
and $D_t^-$ produced by the map $E_t$ \eqref{GenFrontMap} is well-defined.
As shown in \cite[Thm. 1.2]{Barles}, the above statement implies that
\begin{equation}
\underline{u}(x,t) = 1 \text{ in } D_t^+ \text{ and } \bar{u}(x,t)  = -1 \text{ in } D_t^-,
\end{equation}
In other words,
\[
\Bdry\CurBrac{x: \underline{u}(x,t) \geq 1},\,\,\,
\Bdry\CurBrac{x: \bar{u}(x,t) \leq -1} \subset \Gamma_t.
\]
It is in this sense that we say the zero level set of $u_k$ in
the limit moves according to the motion by mean curvature.
Note that in general we can only say the the limiting interface is
\Em{contained in but might not equal to} $\Gamma_t$. See the next item of
remark.

\item
As we are dealing with \Em{discontinuous} initial data and functions
$u_k$, $k\geq 0$, so in general the limit (as $h, \t \Converge 0$)
can be non-unique.
In fact, we can infer using \cite[Thm. 1.1, Cor. 1.3]{Barles} that
we have a unique limit if
\[
\Union_{t>0} \Gamma_t\times\curBrac{t}
= \Bdry\CurBrac{(x,t): u(x,t) > 0}
= \Bdry\CurBrac{(x,t): u(x,t) < 0}
\]
i.e. with no fattening phenomena, and in which case the boundary of the
zero level set for $u_k$ converges to
$\Union_{t>0} \Gamma_t\times\curBrac{t}$ in the sense of Hausdorff
distance.
\end{enumerate}

\subsection{Main Results: Critical and Super-Critical Cases}
We next describe the results in the critical case,
i.e. $\t = \mu h$ for fixed $\mu>0$.
To concentrate on the key ideas of our approach, we assume that locally
near the origin, the boundary of the initial set $\Omega_0$ is represented
by a graph. More specifically, we assume that for some $c_0> 0$,
\label{setgraph}
\Beqn
\Omega_0\intersect(-c_0, c_0)\times(-c_0, c_0)
= \Big\{(x,y): \Abs{x}\leq c_0: -c_0 < y \leq f(x) < c_0\Big\}
\Eeqn
where  $f(x)$ is a $C^2$, even-function satisfying
\Beqn
f(0) = f'(0) = 0,\,\,\,
f''(0) = - \kappa \leq 0.
\Eeqn
The quantity $\kappa$ represents the curvature of $\Bdry\Omega_0$ at $(0,0)$. (See Fig. \setgraph.)
\begin{center}
\includegraphics [width=1.5in, angle = 270]{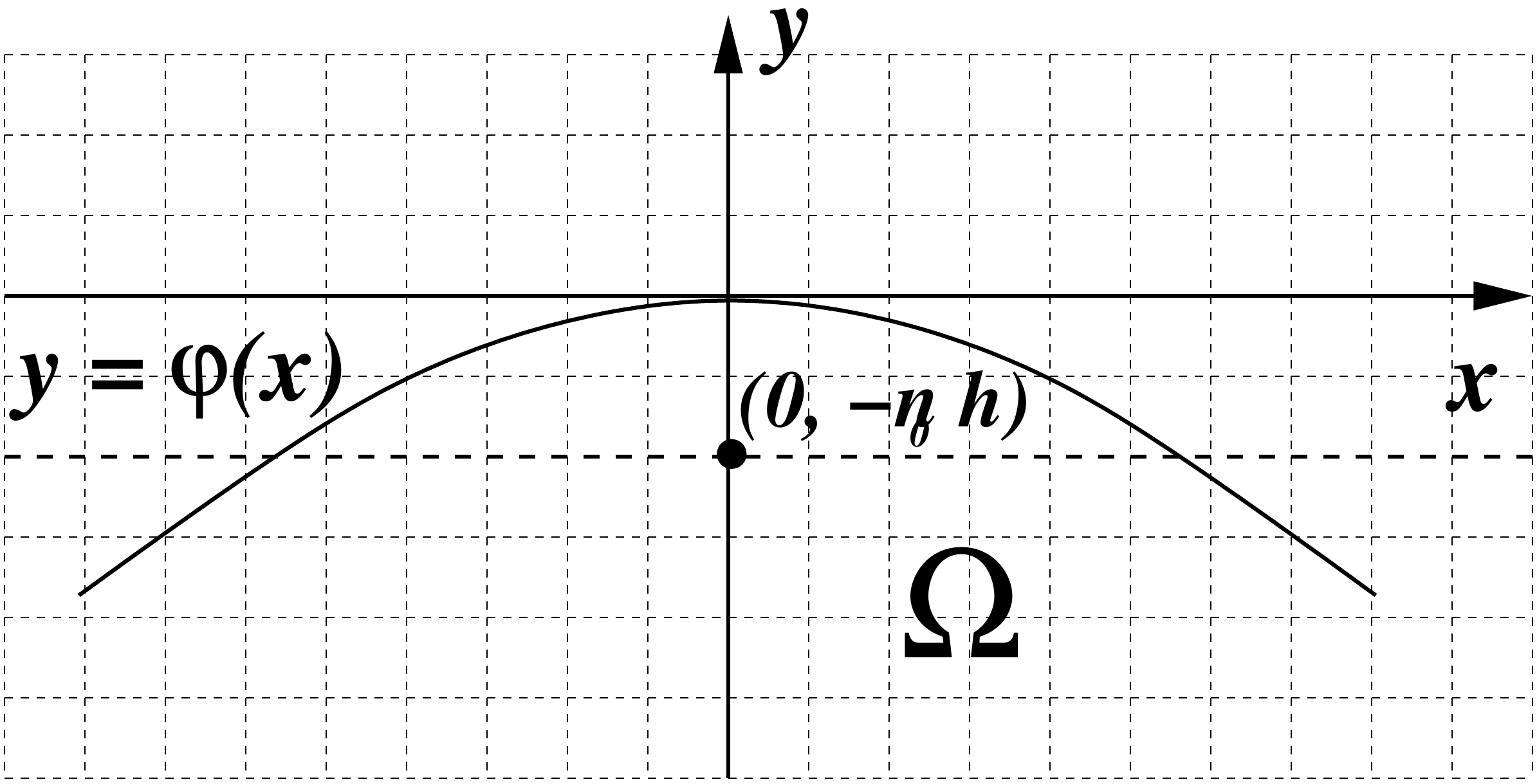}\\
Figure \setgraph.
\end{center}
For simplicity of presentation of the results in this section, we consider an equivalent initialization $\tilde{u}^{m,n}_0:=\frac{1}{2}+ \frac{1}{2} {u}^{m,n}_0$ and the threshholding step $\tilde{u}^{m,n}_{k+1}:=\frac{1}{2}+ \frac{1}{2} {u}^{m,n}_k$, where $\tilde{u}^{m,n}_0$ and $\tilde{u}^{m,n}_{k+1}$ are given by (\ref{scheme-11}) and  (\ref{defn_of_uk}) respectively. In other words, we replace the ``$-1$ $\&$ $1$'' thresholding scheme with a ``$0$ $\&$  $1$'' one. This new scheme will still be denoted with $u^{m,n}_0$ and ${u}^{m,n}_{k+1}$. Consider the function $w$ obtained from \eqref{semidiscrete} for $k=0$.
Let $n_0\in\Z$ be such that
\begin{equation}\label{def_n0}
%\begin{cases}
w^{0,-n_0}(\t) \leq \frac{1}{2},\,\,\,\text{and}\,\,\,
w^{0,-n_0 - 1}(\t) > \frac{1}{2}
%\end{cases}
\end{equation}
For simplicity, we call $n_0$ the ``discrete normal velocity''
of the boundary point $(0,0) \in \d \Omega_0$. The true physical normal velocity
$V$ is defined as:
\Beqn
V
= \lim_{h,\t\converge 0}\frac{n_0 h}{\t}
= \lim_{h,\t\converge 0}\frac{n_0}{\mu}
\Eeqn
We will show that $V$ exists and is still given by a function of the curvature
$\kappa$. Precisely,

%The relation (\ref{def_n0}) implicitly defines $n_0$, in terms of the small parameters $h$ and $t = \mu h$. Thus, the direct computation on $n_0$ from (\ref{def_n0}) requires resolving the problem on a fine scale, which is computationally expensive for small $h$ and may lead to relatively large errors. Our goal is to find the expression for $n_0$ in terms of the (order one) parameters $\kappa$ and $\mu$, thus eliminating the dependence on $h$. The main result in this case
%is the following Theorem

\begin{thm}[Critical case]\label{Th2}
Assume $\kappa\neq 0$.
For sufficiently small $\t>0$, $w^{0,-n}(\t) \leq \frac{1}{2}$ holds
if and only if
\Beqn
\sum_{k = 1}^{n - 1} \int_{0}^{\sqrt{\frac{2k}{\mu \kappa}}} e^{- \frac{x^2}{4}} dx + \frac{1}{2} \int_{0}^{\sqrt{\frac{2 n}{\mu \kappa}}} e^{- \frac{x^2}{4}} dx \leq \frac{1}{2} \int_{\sqrt{\frac{2 n}{\mu \kappa}}}^{\infty} e^{- \frac{x^2}{4}} dx + \sum_{k = n + 1}^{\infty} \int_{ \sqrt{\frac{2 k}{\mu \kappa}}}^{\infty} e^{- \frac{x^2}{4}} dx.
\Eeqn
In particular, if $w^{0,-n_0}(\t) = \frac{1}{2}$, then
\begin{equation}\label{eq_key}
\sum_{k = 1}^{n_0 - 1} \int_{0}^{\sqrt{\frac{2k}{\mu \kappa}}} e^{- \frac{x^2}{4}} dx + \frac{1}{2} \int_{0}^{\sqrt{\frac{2 n_0}{\mu \kappa}}} e^{- \frac{x^2}{4}} dx = \frac{1}{2} \int_{\sqrt{\frac{2 n_0}{\mu \kappa}}}^{\infty} e^{- \frac{x^2}{4}} dx + \sum_{k = n_0 + 1}^{\infty} \int_{ \sqrt{\frac{2 k}{\mu \kappa}}}^{\infty} e^{- \frac{x^2}{4}} dx.
\end{equation}
\end{thm}
Note that discrete velocity $n_0$ is related to $\mu$ and $\kappa$
implicitly. Nevertheless, it is straightforward to see from (\ref{eq_key}) that if $\kappa_1 < \kappa_2$, we have $n_0(\mu, \kappa_1) \leq n_0(\mu, \kappa_2)$.   It can also be seen that if
$\kappa$ is small enough, then $n_0 = 0$, i.e. the interface is
\Em{pinned}. Using monotonicity, the above result
can be extended to the case $\kappa=0$ giving $n_0=0$.

We next show that the result in the critical case is
consistent with the sub-critical case when $\mu \to \infty$:
\begin{thm}\label{Th3}
Let $n_0$, $\mu$ and $\kappa$ satisfy (\ref{def_n0}). Then
\[
\lim_{\mu \to \infty} \frac{n_0}{\mu} = \kappa,
\]
i.e. we get the mean curvature motion as in the subcritical case.
\end{thm}

The main result in the super-critical case follows easily from Theorem 2.
\begin{cor}[Super-critical case]
Assume $\t = \mu h$. If $\mu\kappa$ is sufficiently small,
then $n_0 = 0$, i.e. the front does not move.
\end{cor}
The above statement indicates that the front will not move if either
$\mu$ or $\kappa$ is sufficiently small. From numerical calculation, we
find that the smallness condition is quantified by $\mu \kappa \leq 0.8218$.

Finally, we obtain an extension of Theorem \ref{Th2} to the case of
an anisotropic curvature motion.
In particular, we want to calculate the normal velocity of a boundary
point if the normal line at the point forms an angle with the coordinate
axis. For concreteness, let the normal line at $(0,0)\in\Bdry\Omega$
forms an angle $\theta$ with the $x$-axis (measure in the counter-clockwise
sense). Without loss of generality, we assume
$0 < \theta \leq \frac{\pi}{4}$.
We consider the case that $\tan\theta = \frac{p}{q}$
for some positive integers $p \leq q$.

For simplicity, we assume that locally in the neighborhood of $(0,0)$, the boundary $\d \Omega$ is given by $(x, g(x))$, with $g(x) = \frac{p}{q} x - v_0 x^2$, where $v_0 = \frac{\kappa (q^2 + p^2)^{3/2}}{2 q^{3}}$. This way, the curvature of $\d \Omega$ at $(0,0)$ is $\frac{g^{\prime \prime}}{(1+(g^{'})^2)^{3/2}} = \kappa$.  In this setting the notion of normal motion has to be defined  somewhat differently from the one in Theorem \ref{Th2} since the line in the normal direction intersects the grid only at the points $(np, nq), n \in \Z$, and thus bypasses a number of grid points in between. To make the motion in the normal direction precise, denote $S$ to be the following strip:
\begin{equation}\label{setS}
S:=\{(s,j): 0 \leq s \leq q-1; -\infty < j < \frac{p}{q}s \}.
\end{equation}
For every $(s,j) \in S$, let $d(s,j):= \frac{|jq - sp|}{\sqrt{p^2+q^2}}$ be the distance from $(s,j)$ to the line $y = \frac{p}{q} x$.
Next, we reorder the elements in the set $S$ with respect to $d$ as follows:
\begin{equation}\label{ordering}
S := \{ (s_1, j_1), (s_2, j_2), (s_3, j_3), ... \}
\end{equation}
such that
\[
0 < d_1 < d_2 < d_3 < ..., \text{ where }   d_i:= d(s_i, j_i).
\]
For example, if $p=1$, the points in $S$ will be ordered as
$$\{(1,0), (2,0), ..., (q-1,0), (0,-1), (1,-1), (2,-1), ...,(q-1,-1), (0,-2), (1,-2),...\}.$$
In this setting, we have the following result:
\begin{thm}[Anisotropic curvature motion, critical case]\label{Th4}
Assume $p$, $q$, $\kappa$ and $\mu$ are given, and the set $S$ be ordered as in (\ref{ordering}).
Fix an integer $n_0 \geq 0$. Then, for sufficiently small $t$, $w_{s_{n_0}, j_{n_0}}(\t) \leq \frac{1}{2}$ holds if and only if

\begin{multline}
\frac{1}{2}\int_{0}^{\sqrt{\frac{2 d_{n_0}}{\mu \kappa}}}  e^{-\frac{x^2}{4}} dx + \sum_{i=1}^{n_0-1} \int_{0}^{\sqrt{\frac{2 d_{i}}{\mu \kappa}}} e^{-\frac{x^2}{4}} dx \\
\leq \frac{1}{2}\int_{\sqrt{\frac{2 d_{n_0}}{\mu \kappa}}}^{\infty} e^{-\frac{x^2}{4}} dx + \sum_{i=n_0+1}^{\infty} \int_{\sqrt{\frac{2 d_{i}}{\mu \kappa}}}^{\infty} e^{-\frac{x^2}{4}} dx
\end{multline}

%\begin{multline}
%\frac{1}{2}\int_{0}^{\sqrt{\frac{2 n_0}{\mu \kappa}} \sqrt[4]{p^2+q^2}} e^{-\frac{x^2}{4}} dx + \sum_{l=0}^{q-1} \sum_{m=M_{l}^{-}}^{M_{l}^{+}} \int_{0}^{S_{l,m}} e^{-\frac{x^2}{4}} dx \\
%\leq \frac{1}{2}\int_{\sqrt{\frac{2 n_0}{\mu \kappa}} \sqrt[4]{p^2+q^2}}^{\infty} e^{-\frac{x^2}{4}} dx + \sum_{l=0}^{q-1} \sum_{m= - \infty}^{M_{l}^{-} - 1} \int_{S_{l,m}}^{\infty} e^{-\frac{x^2}{4}} dx
%\end{multline}
%where
%\[
%M_l^{-} := -\left[\frac{p}{q}(n_0 p + l)\right], \ \  M_l^{+} := \left[n_0 q  - \frac{p}{q} l\right]
%\]
%and
%\[
%S_{l,m} = \sqrt{\frac{2}{\mu}} \frac{\sqrt{n_0 q^2 - p l - m q}}{\sqrt{\kappa} (q^2 + p^2)^{\frac{1}{4}}}.
%\]
In particular, if we assume that $w_{s_{n_0}, j_{n_0}}(\t) = \frac{1}{2}$ then

\begin{multline}
\frac{1}{2}\int_{0}^{\sqrt{\frac{2 d_{n_0}}{\mu \kappa}}}  e^{-\frac{x^2}{4}} dx + \sum_{i=1}^{n_0-1} \int_{0}^{\sqrt{\frac{2 d_{i}}{\mu \kappa}}} e^{-\frac{x^2}{4}} dx \\
= \frac{1}{2}\int_{\sqrt{\frac{2 d_{n_0}}{\mu \kappa}}}^{\infty} e^{-\frac{x^2}{4}} dx + \sum_{i=n_0+1}^{\infty} \int_{\sqrt{\frac{2 d_{i}}{\mu \kappa}}}^{\infty} e^{-\frac{x^2}{4}} dx
\end{multline}
In this case, $d_{n_0} h$ is the normal displacement at time $\t$, thus $\frac{d_{n_0}}{\mu}$ is the normal velocity.
\end{thm}

%\begin{remark} \colorRED{Remarks about irrational angle.}
%Since the velocity is essentially described by an integer number of grid points, the velocity in the case, when $\tan \theta$ is not rational is identically the same as the velocity for the angle $\tilde \theta$, which is sufficiently close to $\theta$ and whose tangent is a rational number.
%
%\end{remark}

As the proof of all of our results relies heavily on the asymptotics of discrete
heat kernel, we first collect their key properties and connection to the
continuum heat kernel.

\section{Properties of discrete heat kernels}

\subsection{Derivation of discrete heat kernel and its elementary properties.}
We first consider a one-dimensional analog of (\ref{semidiscrete})
\begin{equation}\label{d_heat_1D}
\begin{cases}
u_\t^n (t) = \frac{1}{h^2}[u^{n+1}(t) + u^{n-1}(t) - 2 u^{n}(t)],
\,\,\,n \in \Z, \,\,\, t \geq 0;\\
u^n(0)  = u_0(n h)
\end{cases}
\end{equation}
where the initial data $u_0$ is an $L^\infty$ function.
The solution of the above is given by
\begin{equation}\label{MBF_soln}
u^m(t) = \sum_{k = -\infty}^{\infty} G_{m-k}\left(\frac{2 t}{h^2}\right) u_0^k
\end{equation}
where
\begin{equation}\label{Greens}
G_n(\a) := \frac{1}{2 \pi}\int_{-\pi}^{\pi} cos(n \xi) e^{\a(cos \xi - 1)} d \xi = e^{-\a} I_{|n|}(\a)
\end{equation}
In the above $I_{|n|}(\a)$ is the {\it Modified Bessel Function}
\[
I_n(\a) = \sum_{m=0}^{\infty} \frac{\left(\frac{\a}{2}\right)^{2m+n}}{m! (m+n)!}, n \geq 0.
\]
In the following we will use the following notation,
\Beqn
\alpha = \frac{2 t}{h^2}
\Eeqn

In order to establish (\ref{MBF_soln}), define the Fourier transform of a sequence $u^m(t)$ as
\[
\hat v (\xi,t):= \frac{1}{\sqrt{2 \pi}} \sum_{m=-\infty}^{\infty} e^{-i m \xi} u^m(t).
\]
Using (\ref{d_heat_1D}), we conclude that
\[
\frac{\d \hat v}{\d t}  = \frac{1}{\sqrt{2 \pi}} \sum_{m=-\infty}^{\infty} e^{-i m \xi} u^m_t(t) = \frac{1}{h^2 \sqrt{2 \pi}} \sum_{m=-\infty}^{\infty} (e^{i \xi} + e^{-i \xi} - 2) e^{- i m \xi} u^m(t) = \frac{2}{h^2}(cos \xi - 1) \hat v
\]
and thus
\[
\hat v(\xi,t) = \hat v(\xi,0) e^{\frac{2 t}{h^2}(cos \xi - 1)}.
\]
Taking the inverse Fourier transform of $\hat v(\xi,t)$, we recover $u^m(t)$:
\[
u^m(t) = \frac{1}{\sqrt{2 \pi}} \int_{- \pi}^{\pi} e^{i m \xi} \hat v(\xi, t) d \xi = \frac{1}{\sqrt{2 \pi}} \int_{- \pi}^{\pi} e^{i m \xi} \hat v(\xi, 0) e^{\frac{2 t}{h^2}(cos \xi - 1)} d \xi = \sum_{k = -\infty}^{\infty} G_{m-k}\left(\frac{2 t}{h^2}\right) u_0 (k h).
\]

Using separation of variables, it is straightforward to verify that the solution to the original problem (\ref{semidiscrete}) (in two dimensions) is given as
\begin{equation}\label{2d_analog}
w^{m,n}(t) =
\sum_{s,j = -\infty}^{\infty}
G_{s-m}\left(\frac{2 t}{h^2}\right)
G_{j-n}\left(\frac{2 t}{h^2}\right) w^{s,j}(0)
\end{equation}

The discrete heat kernel (\ref{Greens}) has the following elementary properties:
%\subsection{Properties of Discrete Heat Kernel.}

\begin{equation}\label{property1}
\sum_{k = -\infty}^{\infty} G_{k}\left(\a\right) = 1
\end{equation}
\begin{equation}\label{property2}
G_{k}\left(\a\right) = G_{-k}\left(\a\right), k \geq 1
\end{equation}
\begin{equation}
\frac{1}{2} G_0(\a) +  \sum_{k = 1}^{\infty} G_k(\a) = \frac{1}{2}
\end{equation}

\subsection{Decay properties.}
The following result is used to control the Green's function outside a fixed
macroscopic domain.
\begin{lem}\label{L3.DecayProp}
Let $G_{n}(\a)$ be the discrete heat kernel \eqref{Greens}
with $\a = \frac{2\t}{h^2}$.
For any fixed $\mu > 0$, suppose $\t \leq \mu h$. Then
\begin{equation}\label{decay_est}
\sum_{k = \left[\frac{3 \mu}{h}\right]}^{\infty} G_{k}(\a) = o\left(\frac{e}{3}\right)^{\frac{\mu}{h}}, h \to 0.
\end{equation}
\end{lem}
\begin{proof}
To simplify the notation, we omit the integer part brackets, denoting  $\left[\frac{3 \mu}{h}\right]$ in (\ref{decay_est}) with $\frac{3 \mu}{h}$. We have
\[
\sum_{k = \frac{3 \mu}{h}}^{\infty} G_{k}(\a) = e^{-\frac{2\t}{h^2}} \sum_{k = \frac{3 \mu}{h}}^{\infty} \sum_{n = 0}^{\infty} \frac{\left(\frac{ \t}{h^2}\right)^{2n+k}}{n! (n+k)!}
\]
Shifting the summation $m = k - \frac{3 \mu}{h}$ and using the elementary inequality $(i+j)! \geq i! j!$, we have
\[
\sum_{k = \frac{3 \mu}{h}}^{\infty} G_{k}(\a) = e^{-\frac{2\t}{h^2}} \sum_{m = 0}^{\infty} \sum_{n = 0}^{\infty} \frac{\left(\frac{\t}{h^2}\right)^{2n+ m + \frac{3 \mu}{h}}}{n! (n+ m + \frac{3 \mu}{h})!}
\leq \frac{\left(\frac{\t}{h^2}\right)^{\frac{3 \mu}{h}}}{\left(\frac{3 \mu}{h}\right)!} e^{-\frac{2\t}{h^2}} \sum_{m = 0}^{\infty} \sum_{n = 0}^{\infty} \frac{\left(\frac{\t}{h^2}\right)^{2n+ m}}{n! (n+ m)!} \leq \frac{\left(\frac{\t}{h^2}\right)^{\frac{3 \mu}{h}}}{\left(\frac{3 \mu}{h}\right)!}
\]
since, by property (\ref{property1}),
\[
e^{-\frac{2\t}{h^2}} \sum_{m = 0}^{\infty} \sum_{n = 0}^{\infty} \frac{\left(\frac{\t}{h^2}\right)^{2n+ m}}{n! (n+ m)!} \leq \sum_{m = -\infty}^{\infty}G_m(\frac{2\t}{h^2}) = 1.
\]
Using the lower bound for the factorial (Stirling approximation
\cite{Abr}),
\[
\left(\frac{3 \mu}{h}\right)! \geq \sqrt{2 \pi \frac{3 \mu}{h}} \left(\frac{3 \mu}{e h }\right)^{\frac{3 \mu}{h}}.
\]
Thus by the assumption $\t \leq \mu h$, we have
\[
\sum_{k = \frac{3 \mu}{h}}^{\infty} G_{k}(\a) \leq \frac{1}{\sqrt{2 \pi \frac{3 \mu}{h}}} \left(\frac{e \t}{3 \mu h}\right)^{\frac{3 \mu}{h}} = o\left(\frac{e}{3}\right)^{\frac{\mu}{h}},
\]
the desired statement
\end{proof}

\subsection{Asymptotic expansions.}
In what follows, we are going to use the asymptotic expansions for modified
Bessel function, \cite[p. 199]{Abr} (see also \cite{watson}):
\begin{itemize}
\item The expansion for $I_\nu(z)$ for any fixed index $\nu$ is given by:
\begin{equation}\label{exp_fixed}
I_\nu(z) = \frac{e^z}{\sqrt{2 \pi z}}\left(1 - \frac{4 \nu^3 - 1}{8 z} + O\left(\frac{1}{z^2}\right)\right),\quad\text{as}\quad
z \to \infty.
\end{equation}
\item Meissel formula \cite{watson}, which holds uniformly for $z \in \R, z > 0$ and $\nu \geq 1$:
\begin{equation}\label{exp_unif}
I_\nu (\nu z) =  \frac{\nu^{\nu}}{e^{\nu} \Gamma(\nu + 1)}\frac{e^{\nu \eta}}{(1 + z^2)^{\frac{1}{4}}} \left(1 + \sum_{k=1}^{\infty} \frac{v_k(\xi)}{\nu^k}\right)
\end{equation}
where
\[
\xi = \frac{1}{\sqrt{1 + z^2}},\,\,\,\,\,\,
\eta = \sqrt{1 + z^2} + \ln \frac{z}{1 + \sqrt{1 + z^2}},
\]
and $v_k(\xi)$ is a polynomial of the form
\begin{equation}\label{vk}
v_k(\xi) = \sum_{s = 0}^k a_s^{(k)} \xi^{k + 2 s},
\end{equation}
defined through the recursive relation $v_0(\xi) = 1$ and
\[
v_{k+1}(\xi) = \frac{1}{2} \xi^2 (1-\xi^2) v_k^{'}(\xi) +  \frac{1}{8} \int_{0}^{\xi} (1 - 5 s^2) v_k(s) ds
\]
\end{itemize}
\begin{prop}
Assume $\t = \t(h)$ is such that $\frac{\t}{h^2} \to \infty$ as $h \to 0$.
Then the following asymptotic expansions hold.
\begin{itemize}
\item For fixed $n \geq 0$,
\begin{equation}\label{asymp1}
G_{n}\left(\frac{2\t}{h^2}\right) = \frac{1}{\sqrt{4 \pi}} \frac{h}{\sqrt{\t}} \left(1 - \frac{h^2}{\t} \frac{4 n^3 - 1}{16} + O\left(\frac{h^4}{\t^2}\right)\right), h \to 0
\end{equation}
\item Let $x \geq \frac{h}{\sqrt{\t}}$. Then there exists an absolute constant $C_0>0$ and $0 \leq C_{h,x} \leq C_0$ such that
\begin{equation}\label{expansion_key}
G_{\frac{\sqrt{\t} x}{h}}\left(\frac{2\t}{h^2}\right) = \frac{1}{\sqrt{4 \pi}} \frac{h}{\sqrt{\t}} e^{-\frac{x^2}{4}}\left(1 + C_{h,x} \frac{h} {x \sqrt{\t}}\right)
\end{equation}
\end{itemize}
\end{prop}
\begin{proof}
Expansion (\ref{asymp1}) follows from \eqref{Greens} and (\ref{exp_fixed}):
\begin{equation}\label{recall_heat_ker}
G_{n}\left(\frac{2\t}{h^2}\right) = e^{-\frac{2\t}{h^2}} I_n\left(\frac{2\t}{h^2}\right)
\end{equation}

For (\ref{expansion_key}), we will make use of the
Meissel formula (\ref{exp_unif}). Borrowing the notations there, we set
\[
\nu = \frac{x \sqrt{\t}}{h},\,\,\,z = \frac{2 \sqrt{\t}}{x h}.
\]
Then
%\[
%\frac{1}{\sqrt{2 \pi \nu}} = \frac{1}{\sqrt{2\pi}} \frac{\sqrt{h}}{\sqrt{x}\sqrt[4]{\t}}
%\]
%\begin{equation}\label{term1}
%\xi = \frac{1}{(1+z^2)^{\frac{1}{4}}} = \frac{\sqrt{x} \sqrt{h}}{(4 t)^{\frac{1}{4}}(1 + \frac{h^2 x^2}{4 t})^{1/4}}
%\end{equation}
\begin{equation}\label{eta}
\eta = \sqrt{1 + z^2} + \ln \frac{z}{1 + \sqrt{1 + z^2}} = \frac{\sqrt{x^2 h^2 + 4 \t}}{x h} + \ln \frac{2 \sqrt{\t}}{h x + \sqrt{x^2 h^2 + 4 \t}}
\end{equation}
so that
\begin{equation}\label{term2}
e^{\eta \nu} = e^{\frac{\sqrt{\t}}{h^2}\sqrt{x^2 h^2 + 4 \t}} \left(\frac{2 \sqrt{\t}}{h x + \sqrt{x^2 h^2 + 4 \t}}\right)^{\frac{x \sqrt{\t}}{h}}
\end{equation}
Finally,
\begin{eqnarray}\label{expansion1}
&& G_{\left[\frac{\sqrt{\t} x}{h}\right]}(\frac{2\t}{h^2}) = e^{-\frac{2\t}{h^2}} I_{\left[\frac{\sqrt{\t} x}{h}\right]}(\frac{2\t}{h^2}) = \\
\nonumber
& = & \frac{\nu^{\nu}}{e^{\nu} \Gamma(\nu + 1)} \frac{\sqrt{x} \sqrt{h}}{(4 \t)^{\frac{1}{4}}} e^{\frac{\sqrt{\t} \sqrt{x^2 h^2 + 4 \t} - 2 \t}{h^2}} \left(\frac{2 \sqrt{\t}}{h x + \sqrt{x^2 h^2 + 4 \t}}\right)^{\frac{x \sqrt{\t}}{h}} \frac{1}{(1 + \frac{h^2 x^2}{4 \t})^{\frac{1}{4}}} \left(1 + \sum_{k=1}^{\infty} \frac{v_k(\xi)}{\nu^k}\right)\\
\nonumber
& = & \frac{\nu^{\nu}}{e^{\nu} \Gamma(\nu + 1)} \frac{\sqrt{x} \sqrt{h}}{(4 \t)^{\frac{1}{4}}}
\cdot \bf{I} \cdot \bf{II} \cdot \bf{III} \cdot \bf{IV}.
\end{eqnarray}

We start with observing that it follows from Stirling approximation
(\cite{Abr}) that for all $\nu \geq 1$:
\[
\frac{\nu^{\nu}}{e^{\nu} \Gamma(\nu + 1)}  =  \frac{1}{\sqrt{2 \pi \nu}}\left(1 + \frac{C_\nu}{\nu}\right)
\]
where $0 \leq C_\nu \leq C_0$ and $c_0$ is an absolute constant. Consequently,
\[
\frac{\nu^{\nu}}{e^{\nu} \Gamma(\nu + 1)} \frac{\sqrt{x} \sqrt{h}}{(4 \t)^{\frac{1}{4}}} = \frac{h}{\sqrt{4 \pi \t}}\left(1 + \frac{C_\nu h}{x \sqrt{\t}}\right)
\]

We now proceed with expanding each term:

{\bf Term I} \MathSty{:= e^{\frac{\sqrt{\t} \sqrt{x^2 h^2 + 4 \t} - 2 \t}{h^2}}}.
Consider
\[
\frac{\sqrt{\t} \sqrt{x^2 h^2 + 4 \t} - 2 \t}{h^2} = \frac{2\t}{h^2}\left(\sqrt{1 + \frac{h^2 x^2}{4 \t}} - 1\right) = \frac{x^2}{4} - \frac{1}{64} \frac{h^2 x^4}{\t} + O\left(\frac{h^4 x^6}{\t^2}\right).
\]
Hence
\[
{\bf I} = e^{\frac{x^2}{4}} e^{-\frac{h^2 x^4}{64 \t}} e^{O\left(\frac{h^4 x^6}{\t^2}\right)} = e^{\frac{x^2}{4}}\left(1 - \frac{h^2 x^4}{64 \t} + O\left(\frac{h^4 x^8}{\t^2}\right)\right)
\]

{\bf Term II} \MathSty{:=\left(\frac{2 \sqrt{\t}}{h x + \sqrt{x^2 h^2 + 4 \t}}\right)^{\frac{x \sqrt{\t}}{h}}} which can be written as
\[
{\bf II}
= \frac{1}{\left(1 + \frac{h x}{2 \sqrt{\t}}\right)^{\frac{2 \sqrt{\t}}{x h} \frac{x^2}{2}}\left(1 + \frac{x^3 h}{8 \sqrt{\t}} + o \left(\frac{x^3 h}{8 \sqrt{\t}}\right)\right)}.
\]
Using the expansion
\[
(1 + y)^{\frac{a}{y}} = e^a - \frac{1}{2} a e^a y + O(e^a a^2 y^2)
\]
we get
\[
\left(1 + \frac{h x}{2 \sqrt{\t}}\right)^{\frac{2 \sqrt{\t}}{x h} \frac{x^2}{2}} = e^{\frac{x^2}{2}} \left(1 - \frac{x^3 h}{8 \sqrt{\t}} + o \left(\frac{x^3 h}{\sqrt{\t}}\right)\right)
\]
Altogether,
\[
{\bf II} = e^{-\frac{x^2}{2}}\Brac{1 + O(\frac{x^6 h^2}{\t})}
\]

{\bf Term III} \MathSty{:=\frac{1}{(1 + \frac{h^2 x^2}{4 \t})^{\frac{1}{4}}}}.
Using Taylor expansion, it becomes:
\[
{\bf III} = 1 - \frac{1}{16} \frac{h^2 x^2}{\t} + O\left(\frac{h^4 x^4}{\t^2}\right)
\]

{\bf Term IV} \MathSty{:=1 + \sum_{k=1}^{\infty} \frac{v_k(\xi)}{\nu^k}}.
In view of (\ref{vk}), for $k \geq 1$
\[
\frac{v_k(\xi)}{\nu^k} = a_0^{(k)} \left(\frac{\xi}{\nu}\right)^k + o\left(\frac{\xi}{\nu}\right)^k.
\]
Since
\[
\xi = \frac{x h}{\sqrt{x^2 h^2 + 4 \t}},
\]
we have
\[
\frac{\xi}{\nu} = \frac{h^2}{\sqrt{x^2 h^2 + 4 \t} \sqrt{\t}} \leq \frac{h^2}{2 \t}
\]
and, therefore,
\[
{\bf IV} =
1 + \frac{1}{8} \frac{\xi}{\nu} + o\left(\frac{\xi}{\nu}\right)
= 1 + O\left(\frac{h^2}{\t}\right).
\]

Combining the four expansions above, (\ref{expansion1}) reads as
\begin{multline}\label{expansion_x}
G_{\left[\frac{\sqrt{\t} x}{h}\right]}(\frac{2\t}{h^2}) = \frac{h}{\sqrt{4 \pi}} \frac{1}{\sqrt{\t}} e^{-\frac{x^2}{4}}\times\\
\left(1 + C_\nu \frac{h} {x \sqrt{\t}}\right)\left(1 + O\left(\frac{h^2 x^4}{\t}\right)\right)\left(1 + O\left(\frac{h^2 x^6}{\t}\right)\right)\left(1 + O\left(\frac{h^2}{\t}\right)\right)
\end{multline}
for $x \geq \frac{h}{\sqrt{\t}}$, which gives the expansion (\ref{expansion_key}) up to $O\left(\frac{h}{\sqrt{\t}}\right)$ order terms.
\end{proof}

\section{Proof of Theorem \ref{P1.Case1} - Sub-Critical Case.}

We are going to prove that the function $\bar{u}$, defined in (\ref{supsol}), is a viscosity subsolution of the equation (\ref{meancurv}).
The fact that (\ref{subsol}) is a supersolution of the same equation
can be established analogously.
The strategy of proof follows \cite{Barles} closely.
Essentially we will demonstrate how to reduce the discrete computation to
\cite[Prop. 4.1]{Barles} which is the key consistency result needed.

\noindent
\Bf{Step I.}
First let $\varphi(x, t)$ be a smooth test function satisfying
  \begin{equation}\label{coercivity}
  \lim_{|x| + t \to \infty} \varphi(x, t) = +\infty.
  \end{equation}

Denote $(\bar{x}, \bar t)$ to be a strict global maximum point of
$\bar{u} - \varphi$. If $\bar{u} (\bar{x}, \bar t) = -1$, then $\hat{u} = -1$ in a neighborhood of $(\bar{x}, \bar t)$ since $\hat{u}$ is upper semicontinuous and takes values in $\{-1,1\}$. Consequently, $D \varphi(\bar{x},\bar{t}) = 0$ and $\frac{\d \varphi}{\d \t}(\bar{x}, \bar t) = 0$, thus by (\ref{consistency2sub}) the conclusion follows. Similar reasoning yields the desired result in the case when $(\bar{x}, \bar t)$ is in the interior of the set $\{\bar{u} = 1\}$. Therefore, without loss of generality, assume that $(\bar{x}, t)$ is at the boundary of the set $\{\bar{u} = 1\}$. We claim that there exists a sequence $\{(m_h h, n_h h); k_{\t} \t\}$, converging to $(\bar{x}=(\bar{x}_1, \bar{x}_2); \bar t)$, such that
\begin{equation}\label{existmax}
u^{m_h, n_h}_{k_{\t}} - \varphi(m_h h, n_h h; k_{\t} \t) = \max_{\Z^2 \times \N}(u^{m, n}_{k} - \varphi(m h, n h; k \t))
\end{equation}
and
\begin{equation}\label{lim1}
\lim_{h \to 0} u^{m_h, n_h}_{k_{\t}} = 1
\end{equation}
Indeed, it follows from the coercivity condition (\ref{coercivity}) that $\max_{\Z^2 \times \N}(u^{m, n}_{k} - \varphi(m h, n h; k \t))$ is attained at some $m = m_h$, $n = n_h$ an $k = k_{\t}$, and up to a subsequence,
there exist $\bar{x} = (\bar{x}_1, \bar{x}_2)$ and $\bar{t}$ such that $h m_h \to \bar{x}_1$, $h n_h \to \bar{x}_2$ and $\t k_{\t} \to \bar{t}$. This point $(\bar{x}, \bar{t})$ must be the global maximum of $\bar{u} - \varphi$.
In addition, (\ref{lim1}) holds, for otherwise it will contradict
(\ref{existmax}) and the definition of $\bar{u}$.

Since $u^{m_h, n_h}_{k_{\t}}$ takes values only $1$ and $-1$, it follows from (\ref{lim1}) that $u^{m_h, n_h}_{k_{\t}} \equiv 1$ for sufficiently small $h$. For such $h$, the fact that $(m_h, n_h; k_{\t})$ is a maximum point implies that for all $m,n$ and $k$ in $\Z$, we have
\[
u^{m, n}_{k} \leq 1 - \varphi(m_h h, n_h h; k_{\t} \t) + \varphi(mh, nh; k)
\]
Using the same reasoning in \cite[p. 490]{Barles},
we conclude that the above inequality leads to
\[
S^h(\t)\left[
{\rm sign}^{*}\Big(
\f(\cdot, (k_{\t} - 1)\t) - \f(m_h h, n_h h; k_{\t} \t)
\Big)
\right](m_h h, n_h h) \geq 0
\]
where $\text{sign}^*$ is the upper semi-continuous envelope of the
sign function. The above is equivalent to
\[
S^h(\t)\left[\left(\frac{1+ {\rm sign}^{*}}{2}\right)
\Big(
\f(\cdot, (k_{\t} - 1)\t) - \f(m_h h, n_h h; k_{\t} \t)
\Big)
\right](m_h h, n_h h) \geq \frac{1}{2}.
\]
In terms of discrete kernels, the above inequality reads as
\begin{equation}\label{ineq1}
\sum_{(m,n) \in Q_h} G_{m - m_h}(\a) G_{n - n_h}(\a)  \geq \frac{1}{2},
\,\,\,\,\,\,\alpha=\frac{2\t}{h^2}
\end{equation}
where
\begin{equation}\label{Qh}
Q_h = \{(m,n) \in \Z^2: \f(m h, n h; (k_{\t} - 1) \t) - \f(m_h h, n_h h; k_{\t} \t) \geq 0 \}
\end{equation}

\noindent
\Bf{Step II.} In this step, we will show how \eqref{ineq1} leads to
\eqref{consistency2sub} and \eqref{consistency2super}. We first express the
left hand side of \eqref{ineq1} as an integral. For this, we write
\Beqn\label{ineq1.RS}
\sum_{(m,n) \in Q_h} G_{m - m_h}(\a) G_{n - n_h}(\a) = \int_{\tilde{Q}_h} G_{[\tilde x_1]}(\alpha) G_{[\tilde x_2]} (\a) d \tilde x_1  \, d \tilde x_2,
\Eeqn
with
\[
\tilde{Q}_h := \{(\tilde x_1, \tilde x_2) \in \R^2: ([\tilde x_1] - m_h, [\tilde x_2] - n_h) \in Q_h\}
\]
Next, introducing $x_1 = \frac{h \tilde x_1}{\sqrt{\t}}$ and $x_2 = \frac{h \tilde x_2}{\sqrt{\t}}$, we rewrite \eqref{ineq1.RS} as
\begin{equation}\label{for1/2}
\frac{\t}{h^2} \int_{\hat{Q}_h} G_{\left[\frac{\sqrt{\t} x_1}{h}\right]}(\alpha) G_{\left[\frac{\sqrt{\t} x_2}{h}\right]} (\a) d  x_1  \, d x_2 \geq  \frac{1}{2},
\end{equation}
where
\Beqn
\hat{Q}_h = \frac{h}{\sqrt{\t}} \tilde{Q}_h = \{(x_1, x_2) \in \R^2: \left(\left[\frac{\sqrt{\t}x_1}{h}\right] - m_h, \left[\frac{\sqrt{\t}x_2}{h}\right] - n_h\right) \in Q_h\}
\Eeqn

\Em{Claim.}
For any $G \subset \R^2$, $f: G \to \R$ and $\delta > 0$ denote $RS_\dl[f,G]$ to be the Riemann sum for $f$ in $G$ with step $\dl$, i.e.
\[
RS_\dl[f,G] : = \dl^2 \sum_{(m,n)\in\Z^2,(m\dl,n\dl) \in G} f(m\dl, n\dl).
\]
For $f = e^{-\frac{|x|^2}{4}}$ and $G \subset \R^2$ we have the following error estimate between the Riemann sum and the integral:
\begin{equation}\label{RSest}
\left|RS_{\delta}[e^{-\frac{|x|^2}{4}}, G] - \int_{G} e^{-\frac{|x|^2}{4}} dx \right| \leq C \delta
\end{equation}
for some $C>0$ which does not depend on $G$ and $\delta$. The estimate (\ref{RSest}) follows from a 1D inequality
\begin{equation}\label{RSest1D}
\left|RS_{\delta}[e^{-\frac{x^2}{4}}, [a,b]] - \int_{a}^{b} e^{-\frac{x^2}{4}} dx \right| \leq 2 \delta
\end{equation}
which holds uniformly for all $-\infty \leq a < b \leq + \infty$.

Now we proceed to analyze \eqref{for1/2}. Due to (\ref{expansion_key}), we have
\begin{eqnarray}
& & \frac{\t}{h^2} \int_{\hat{Q}_h} G_{\left[\frac{\sqrt{\t} x_1}{h}\right]}(\alpha) G_{\left[\frac{\sqrt{\t} x_2}{h}\right]} (\a) d  x_1  \, d x_2  \nonumber\\
& = & RS_{\frac{h}{\sqrt{\t}}}\frac{1}{4 \pi}[e^{\frac{-x_1^2 - x_2^2}{4}}, \hat{Q}_h]
+ C_1 \frac{h}{\sqrt{\t}} RS_{\frac{h}{\sqrt{\t}}}[\frac{1}{|x_1|}e^{\frac{-x_1^2 - x_2^2}{4}}, \hat{Q}_h] + C_2 \frac{h}{\sqrt{\t}} RS_{\frac{h}{\sqrt{\t}}}[\frac{1}{|x_2|}e^{\frac{-x_1^2 - x_2^2}{4}}, \hat{Q}_h]\nonumber\\
& & + o\left(\frac{h}{\sqrt{\t}}\right).\label{xxx}
\end{eqnarray}
For $i = 1,2$, by \eqref{asymp1}:
\begin{equation}\label{errest1}
RS_{\frac{h}{\sqrt{\t}}}[\frac{1}{|x_i|}e^{\frac{-x_1^2 - x_2^2}{4}}, \hat{Q}_h] \leq  \int_{\R^2 \setminus [-\frac{h}{\sqrt{\t}}, \frac{h}{\sqrt{\t}}]^2} \frac{1}{|x_i|} e^{\frac{-x_1^2 - x_2^2}{4}} d x_1 \, d x_2 + O\left(\frac{h}{\sqrt{\t}}\right) = O\left(\ln \frac{h}{\sqrt{\t}}\right).
\end{equation}
Applying \eqref{RSest} to (\ref{xxx}) and making use of (\ref{errest1}),
inequality (\ref{for1/2}) becomes
\begin{equation}\label{yyy}
\frac{1}{4 \pi} \int_{\hat{Q}_h} e^{\frac{-x_1^2 - x_2^2}{4}} d x_1 \, d x_2 + M_h \geq \frac{1}{2},
\end{equation}
with
\begin{equation}\label{Log}
M_h = O\left(\frac{h}{\sqrt{\t}} \ln \frac{h}{\sqrt{\t}}\right).
\end{equation}
%After re-scaling, (\ref{yyy}) may be equivalently written as
%\begin{equation}\label{zzz}
%\frac{1}{4 \pi t} \int_{Q^{disc}_h} e^{\frac{-(x_1-m_h h)^2 - (x_2-n_h h)^2}{4 t}} dx_1 \, d x_2 + M_h \geq \frac{1}{2},
%\end{equation}
%where
%\[
%Q^{disc}_h : = \{(x_1, x_2) \in \R^2, (x_1,x_2) \in [mh, (m+1)h) \times [nh, (n+1)h), (m,n) \in Q_h\}
%\]
%and $Q_h$ is defined in (\ref{Qh}).
Roughly speaking $\hat{Q}_h$ is a discretized set. In order to apply \cite[Prop. 4.1]{Barles} we need to obtain an estimate (\ref{yyy}) in which $\hat{Q}_h$ is replaced with its continuum analog with well-controlled error.
Consider
%\[
%Q^{cont}_h : = \{(x_1, x_2) \in \R^2, \f(x_1, x_2; (k_t - 1) t) - \f(m_h h, n_h h; k_t t) \geq 0 \}
%\]
\[
Q^{cont}_h : = \{(x_1, x_2) \in \R^2, \f(\sqrt{\t} x_1 - m_h h, \sqrt{\t} x_2 - n_h h; (k_{\t} - 1) \t) - \f(m_h h, n_h h; k_{\t} \t) \geq 0 \}
\]
Since $\f$ is a smooth function which satisfies the growth condition (\ref{coercivity}), $\partial \hat{Q}_h$ is bounded. Moreover, if $x_1^* \in \R$  is chosen s.t. $\frac{\sqrt{\t}}{h} x_1^* \in \Z$, $f(y):=\left[\frac{\sqrt{\t} y}{h}\right]$ is constant for $y \in [x_1^*, x_1^* + \frac{h}{\sqrt{\t}})$. This way, for any $(x_1, x_2) \in \d \hat{Q}_h$ we have ${\rm dist } [(x_1, x_2), \d Q^{cont}_h] \leq \frac{h}{\sqrt{\t}} $. This enables us to say that
\[
|(Q^{cont}_h \setminus \hat{Q}_h) \cup (\hat{Q}_h \setminus Q^{cont}_h) | \leq C\frac{h}{\sqrt{\t}},
\]
which, in turn, implies that

%If both $Q^{cont}_h$ and $Q^{disc}_h$ are bounded,
%then
%$|(Q^{cont}_h \setminus Q^{disc}_h) \cup (Q^{disc}_h \setminus Q^{cont}_h) | \leq Ch$. Hence we have,
%\begin{equation}\label{boundary_layer}
%\left|\frac{1}{4 \pi t} \int_{Q^{cont}_h} e^{\frac{-(x_1-m_h h)^2 - (x_2-n_h h)^2}{4 t}} d x_1 \, d x_2 - \frac{1}{4 \pi t} \int_{Q^{disc}_h} e^{\frac{-(x_1-m_h h)^2 - (x_2-n_h h)^2}{4 t}} dx_1 \, d x_2 \right| \leq C \frac{h}{\sqrt{t}}
%\end{equation}
\begin{equation}\label{boundary_layer}
\left|\frac{1}{4 \pi} \int_{\hat{Q}_h} e^{\frac{-x_1^2 - x_2^2}{4}} d x_1 \, d x_2 - \frac{1}{4 \pi} \int_{Q^{cont}_h} e^{\frac{-x_1^2 - x_2^2}{4}} d x_1 \, d x_2 \right| \leq C \frac{h}{\sqrt{\t}}
\end{equation}
with $C$ independent on $h$.
%In the contrary case, if either of those sets is not bounded, for fixed $R>0$
%\[
%\left|\frac{1}{4 \pi t} \int_{Q^{cont}_h} e^{\frac{-(x_1-m_h h)^2 - (x_2-n_h h)^2}{4 t}} d x_1 \, d x_2 - \frac{1}{4 \pi t} \int_{Q^{disc}_h} e^{\frac{-(x_1-m_h h)^2 - (x_2-n_h h)^2}{4 t}} dx_1 \, d x_2 \right| \leq
%\]
%\[
%\left|\frac{1}{4 \pi t} \int_{Q^{cont}_h \cap B_R(0)} e^{\frac{-(x_1-m_h h)^2 - (x_2-n_h h)^2}{4 t}} d x_1 \, d x_2 - \frac{1}{4 \pi t} \int_{Q^{disc}_h \cap B_R(0)} e^{\frac{-(x_1-m_h h)^2 - (x_2-n_h h)^2}{4 t}} dx_1 \, d x_2 \right| + O(e^{-\frac{R^2}{t}})
%\]
%so that (\ref{boundary_layer}) holds as well.
After rescaling, we have
\Beqn\label{resc}
\frac{1}{4 \pi} \int_{Q^{cont}_h} e^{\frac{-x_1^2 - x_2^2}{4}} d x_1 \, d x_2 = \frac{1}{4 \pi \t} \int_{Q^{resc}_h} e^{\frac{-(x_1-m_h h)^2 - (x_2-n_h h)^2}{4 \t}} d x_1 \, d x_2
\Eeqn
where
\[
Q^{resc}_h : = \{(x_1, x_2) \in \R^2, \f(x_1, x_2; (k_{\t} - 1) \t) - \f(m_h h, n_h h; k_{\t} \t) \geq 0 \}
\]
Combining (\ref{yyy}) with (\ref{boundary_layer}) and (\ref{resc}), we have
\[
0 \leq \frac{1}{2} - \sum_{(m,n) \in Q_h} G_{m - m_h}(\a) G_{n - n_h}(\a) = \frac{1}{2} - \frac{1}{4 \pi \t} \int_{Q^{resc}_h} e^{\frac{-(x_1-m_h h)^2 - (x_2-n_h h)^2}{4 \t}} dx_1 \, dx_2 + M_h.
\]

Next we consider two cases.
If $D\phi((\bar{x}_1,\bar{x}_2);\bar{t}) \neq 0$, we apply the result of
\cite[Prop. 4.1]{Barles} with $\phi_{h}((x_1,x_2);t) := \phi((x_1,x_2);t) - \phi((m_h h, n_h h); k_{\t} \t)$:
\[
0 \geq \liminf_{h} \frac{1}{\sqrt{\t}} \left(\frac{1}{2} - \sum_{(m,n) \in Q_h} G_{m - m_h}(\a) G_{n - n_h}(\a) \right) =
\]
\[
 = \liminf_{h} \frac{1}{\sqrt{\t}} \left(\frac{1}{2} - \frac{1}{4 \pi \t} \int_{Q^{resc}_h} e^{\frac{-(x_1-m_h h)^2 - (x_2-n_h h)^2}{4 \t}} dx_1 \, dx_2 + M_h\right) \geq
\]
\[
\geq \frac{1}{2 \sqrt{\pi}|D \phi((\bar{x}_1,\bar{x}_2),t)|}\left(\frac{\d \phi}{\d t}  - \Delta \phi + \frac{(D^2 \phi D \phi | D \phi)}{|D\phi|^2}\right)((\bar{x}_1,\bar{x}_2), \bar{t})
\]
which yields the desired result.

It remains to show that in the case when $D\phi((\bar{x}_1,\bar{x}_2);\bar{t}) = 0$, $D^2\phi((\bar{x}_1,\bar{x}_2);\bar{t}) = 0$, and
\begin{equation}\label{qqq}
\frac{1}{2} - \frac{1}{4 \pi \t} \int_{Q^{resc}_h} e^{\frac{-(x_1-m_h h)^2 - (x_2-n_h h)^2}{4 \t}} dx_1 \, dx_2 + M_h \leq 0
\end{equation}
with $M_h$ satisfying (\ref{Log}) for sufficiently small $h$, we have
$\frac{\partial \phi}{\partial t}((\bar{x_1},\bar{x_2}); \bar{t}) \leq 0$. The condition (\ref{qqq}), strictly speaking, is different from the one in Proposition 4.1 \cite{Barles}, yet enables us to apply the technique developed in \cite{Barles}. In particular, three cases are possible:
\begin{enumerate}
\item[I.] Along some subsequence $|D \phi(m_h h, n_h h; k_{\t} \t)| \neq 0$, and $\frac{\sqrt{\t}}{|D \phi(m_h h, n_h h; k_{\t} \t)|} \to 0$;
\item[II.] Along some subsequence $|D \phi(m_h h, n_h h; k_{\t} \t)| \equiv 0$ or $\frac{\sqrt{\t}}{|D \phi(m_h h, n_h h; k_{\t} \t)|} \to \infty$;
\item[III.] Along some subsequence $|D \phi(m_h h, n_h h; k_{\t} \t)| \neq 0$, and $\frac{\sqrt{\t}}{|D \phi(m_h h, n_h h; k_{\t} \t)|} \to l > 0$.
\end{enumerate}

In Case I, arguing as in \cite{Barles}, we deduce that
\begin{multline}\label{aaa}
\frac{|D \phi(m_h h, n_h h; k_{\t} \t)|}{\sqrt{\t}} \left(\frac{1}{2} - \frac{1}{4 \pi \t} \int_{Q^{resc}_h} e^{\frac{-(x_1-m_h h)^2 - (x_2-n_h h)^2}{4 \t}} dx_1 \, dx_2\right) \geq \\
\geq \frac{1}{2 \sqrt{\pi}} \left(\frac{\d \phi}{\d t}(x_1, x_2; t) - c\right) + o(1)
\end{multline}
for any $c>0$. However, since $$\frac{|D \phi(m_h h, n_h h; k_{\t} \t)|}{\sqrt{\t}} M_h = o\left(\frac{h}{\t} \ln \frac{h}{\sqrt{\t}} \right) \to 0,$$
passing to the limit in (\ref{aaa}) as $h \to 0$, using (\ref{qqq}), we have $\frac{\d \phi}{\d t}(x_1, x_2; t) - c \leq 0$ for all $c>0$, which yields the desired result.

The conclusion in the Cases II and III follows from Proposition 4.1 \cite{Barles} without any change.

\section{Motion in the critical case: $\t = \mu h$.}
This section provides the proofs for the statements related to the
critical case.

\subsection{Proof of Theorem \ref{Th2}}
The results will follow from the following statements.
\begin{itemize}
\item[(i)] \Bf{(Lower bound)}
If $w^{0,-n_0}(\tau) \geq \frac{1}{2}$ we have
\Beqn\label{lower.bd}
\sum_{k = 1}^{n_0 - 1} \int_{0}^{\sqrt{\frac{2k}{\mu \kappa}}} e^{- \frac{x^2}{4}} dx + \frac{1}{2} \int_{0}^{\sqrt{\frac{2 n_0}{\mu \kappa}}} e^{- \frac{x^2}{4}} dx \geq \frac{1}{2} \int_{\sqrt{\frac{2 n_0}{\mu \kappa}}}^{\infty} e^{- \frac{x^2}{4}} dx + \sum_{k = n_0 + 1}^{\infty} \int_{ \sqrt{\frac{2 k}{\mu \kappa}}}^{\infty} e^{- \frac{x^2}{4}} dx
\Eeqn
\item[(ii)] \Bf{(Upper bound)}
If $w^{0,-n_0}(\tau) \leq \frac{1}{2}$, we have
\Beqn\label{upper.bd}
\sum_{k = 1}^{n_0 - 1} \int_{0}^{\sqrt{\frac{2k}{\mu \kappa}}} e^{- \frac{x^2}{4}} dx + \frac{1}{2} \int_{0}^{\sqrt{\frac{2 n_0}{\mu \kappa}}} e^{- \frac{x^2}{4}} dx \leq \frac{1}{2} \int_{\sqrt{\frac{2 n_0}{\mu \kappa}}}^{\infty} e^{- \frac{x^2}{4}} dx + \sum_{k = n_0 + 1}^{\infty} \int_{ \sqrt{\frac{2 k}{\mu \kappa}}}^{\infty} e^{- \frac{x^2}{4}} dx
\Eeqn
\end{itemize}

We refer to page \pageref{setgraph} for the notation concerning the set
$\Omega$.
Assume $n_0 \geq 0$ is such that $w^{0,-n_0}(\tau) \geq \frac{1}{2}$. By Lemma 1,
\[
\frac{1}{2} \leq w^{0,-n_0}(\t) = \sum_{(sh,jh) \in \Omega}  G_{s} (\a) G_{j + n_0}(\a) = \sum_{s = -\infty}^{\infty} \sum_{j = -\infty}^{\left[\frac{\f (sh)}{h}\right]} G_{s} (\a) G_{j + n_0}(\a) + o\left(\frac{e}{3}\right)^{\frac{1}{h}}.
\]
Having shifted the summation indices, we get
\begin{equation}\label{sum to int 1}
\frac{1}{2} \leq \sum_{s = -\infty}^{\infty} \sum_{j = -\infty}^{\left[\frac{\f (sh)}{h}\right] + n_0} G_{s} (\a) G_{j}(\a) + o\left(\frac{e}{3}\right)^{\frac{1}{h}}.
\end{equation}
Similarly, if $w^{0,-n_0}(\tau) \leq \frac{1}{2}$, we have
\begin{equation}\label{sum to int 2}
\frac{1}{2} \geq \sum_{s = -\infty}^{\infty} \sum_{j = -\infty}^{\left[\frac{\f (sh)}{h}\right] + n_0} G_{s} (\a) G_{j}(\a) + o\left(\frac{e}{3}\right)^{\frac{1}{h}}.
\end{equation}
We briefly outline the idea of the proof. The relations (\ref{sum to int 1}) and (\ref{sum to int 2}) implicitly define $n_0$ in terms of $h$.
Our goal is to obtain the description of $n_0$, independent of $h$. To this end, we are going to establish the asymptotic (in $h$) expansion of the sum in the right hand side of (\ref{sum to int 1}) and (\ref{sum to int 2}):
\begin{equation}\label{summation}
\sum_{s = -\infty}^{\infty} \sum_{j = -\infty}^{\left[\frac{\f (sh)}{h}\right] + n_0} G_{s} (\a) G_{j}(\a) = \frac{1}{2} + \frac{1}{\sqrt{4 \pi}} \Phi(n_0, \mu, \kappa) \sqrt{h} + o(\sqrt{h}), h \to 0
\end{equation}
where
\Beqn\label{Psi}
\Phi(n_0, \mu, \kappa):= \sum_{k = 1}^{n_0 - 1} \int_{0}^{\sqrt{\frac{2k}{\mu \kappa}}} e^{- \frac{x^2}{4}} dx + \frac{1}{2} \int_{0}^{\sqrt{\frac{2 n_0}{\mu \kappa}}} e^{- \frac{x^2}{4}} dx - \frac{1}{2} \int_{\sqrt{\frac{2 n}{\mu \kappa}}}^{\infty} e^{- \frac{x^2}{4}} dx - \sum_{k = n_0 + 1}^{\infty} \int_{ \sqrt{\frac{2 k}{\mu \kappa}}}^{\infty} e^{- \frac{x^2}{4}} dx,
\Eeqn
and the desired result follows.

It is convenient to group the indices in the summation (\ref{summation}) into the
following sets:
\[
I_1:=\{(s,j): 0 < jh < \varphi(sh) + n_0  h\},
\]
\[
I_2:=\{(s,j): \varphi(sh) + n_0 h \leq jh < 0\},
\]
\[
I_3:=\{(s,j): j h < \min \{\varphi(sh) + n_0 h, 0\}\}.
\]
Denote also
\[
b I_1:=\{(s,0): 0 < \varphi(sh) + n_0 h\},
\]
\[
b I_2:=\{(s,0): 0 \geq \varphi(sh) + n_0 h\}.
\]
(See Fig. \griddisect.)

\begin{center}
\includegraphics [width=1.8in, angle=270]{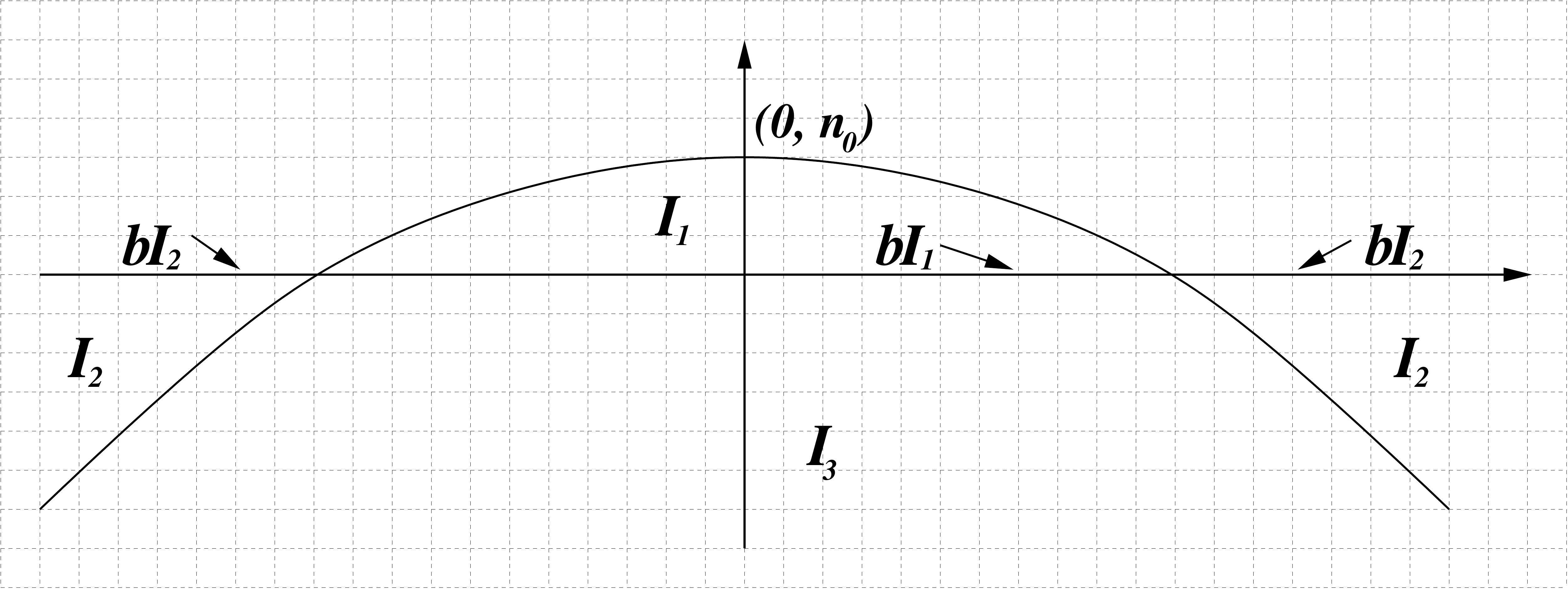}\\
Figure \griddisect.
\end{center}

%\colorRED{Some explanation of the strategy.}

With the above, we have
\[
\sum_{s = -\infty}^{\infty} \sum_{j = -\infty}^{\left[\frac{\f (sh)}{h}\right] + n_0} G_{s} (\a) G_{j}(\a) = \sum_{(s,j) \in I_1} G_{s} (\a) G_{j}(\a) + \sum_{(s,j) \in I_3} G_{s} (\a) G_{j}(\a) + \sum_{(s,j) \in b I_1} G_{s} (\a) G_{j}(\a)
\]
On the other hand, it follows from the properties of discrete heat kernels that
\begin{equation}\label{in1}
\sum_{(s,j) \in I_2} G_{s} (\a) G_{j}(\a) + \sum_{(s,j) \in I_3} G_{s} (\a) G_{j}(\a) + \frac{1}{2}\sum_{(s,j) \in b I_1} G_{s} (\a) G_{j}(\a) + \frac{1}{2}\sum_{(s,j) \in b I_2} G_{s} (\a) G_{j}(\a) = \frac{1}{2}
\end{equation}
In view of (\ref{sum to int 1}),
\begin{equation}\label{in2}
\sum_{(s,j) \in I_1} G_{s} (\a) G_{j}(\a) + \sum_{(s,j) \in I_3} G_{s} (\a) G_{j}(\a) + \sum_{(s,j) \in b I_1} G_{s} (\a) G_{j}(\a) \geq \frac{1}{2} + o\left(\frac{e}{3}\right)^{\frac{1}{h}}
\end{equation}
Subtracting (\ref{in2}) from (\ref{in1}), we obtain
\begin{equation}\label{pre_lower}
\sum_{(s,j) \in I_1} G_{s} (\a) G_{j}(\a) + \frac{1}{2} \sum_{(s,j) \in b I_1} G_{s} (\a) G_{j}(\a) \geq \sum_{(s,j) \in I_2} G_{s} (\a) G_{j}(\a) + \frac{1}{2}\sum_{(s,j) \in  I_2} G_{s} (\a) G_{j}(\a) + o\left(\frac{e}{3}\right)^{\frac{1}{h}}
\end{equation}

Now for a fixed $n \geq 1$, let $s(n) > 0$ be the integer satisfy
\begin{equation*}\label{s(n)}
\left[\frac{\f (s(n) h)}{h}\right] = - n.
\end{equation*}
Since
\[
\f(sh) = \f(0) + sh \f^{'}(0) + \frac{(sh)^2}{2} \f^{''}(0) + O(h^3) = -\frac{\kappa}{2} (sh)^2 + O(h^3),
\]
we have
\begin{equation}\label{expr}
s(n) = \left[\sqrt{\frac{2 n}{\kappa h}}\right].
\end{equation}

Changing the order of summation in (\ref{pre_lower}) using the notation (\ref{expr}), we have if $w^{0,-n_0}(\tau) \geq \frac{1}{2}$, then
\begin{multline}\label{identity1}
\sum_{j=1}^{n_0} G_j(\a)\sum_{|s| \leq s(n_0-j)} G_s(\a) + \frac{1}{2}G_0(\a)\sum_{|s| \leq s(-n_0)}G_s(\a) \geq \\
\frac{1}{2}G_0(\a)\sum_{|s| \geq s(-n_0)+1}G_s(\a) + \sum_{j = -\infty}^{-1} G_j(\a) \sum_{|s| > s(n_0 - j + 1)} G_s(\a) + o\left(\frac{e}{3}\right)^{\frac{1}{h}}
\end{multline}
Using the same reasoning, we get that if $w^{0,-n_0}(\t) \leq \frac{1}{2}$, we have
\begin{multline}\label{identity2}
\sum_{j=1}^{n_0} G_j(\a)\sum_{|s| \leq s(n_0-j)} G_s(\a) + \frac{1}{2}G_0(\a)\sum_{|s| \leq s(-n_0)}G_s(\a) \leq \\
\frac{1}{2}G_0(\a)\sum_{|s| \geq s(-n_0)+1}G_s(\a) + \sum_{j = -\infty}^{-1} G_j(\a) \sum_{|s| > s(n_0 - j + 1)} G_s(\a) + o\left(\frac{e}{3}\right)^{\frac{1}{h}}
\end{multline}

Roughly speaking, the main technical difficulty in establishing the asymptotic behavior of the terms in (\ref{identity2}) is caused by the term
\[
\sum_{j = -\infty}^{-1} G_j(\a) \sum_{|s| > s(n_0 - j + 1)} G_s(\a)
\]
since it involves an infinite sum of the products of Bessel functions and Riemann sums. To proceed, we are going to estimate this term from below using finitely many such products, thus obtaining the lower bound \eqref{lower.bd}. On the other hand, we are going to estimate this term from above, using a priori bounds for Bessel functions, and get the upper bound \eqref{upper.bd}. Once we verify that those bounds match, we will obtain the asymptotic expansion of all terms in (\ref{identity2}).

\subsection{The lower bound \eqref{lower.bd}.}

Let $w^{0,-n_0}(\tau) \geq \frac{1}{2}$. Fix an arbitrary $N \geq 1$. It follows from (\ref{identity2}) that
\begin{multline}\label{identityLB}
\sum_{j=1}^{n_0}G_j(\a)\sum_{|s| \leq s(n_0-j)} G_s(\a) + \frac{1}{2}G_0(\a)\sum_{|s| \leq s(-n_0)}G_s(\a) \geq \\
\frac{1}{2}G_0(\a)\sum_{|s| \geq s(-n_0)+1}G_s(\a) + \sum_{j = -N}^{-1} G_j(\a) \sum_{|s| > s(n_0 - j + 1)} G_s(\a)
\end{multline}
%\[
%\sum_{|s|< S_{-1}}G_s(\a)\sum_{j=1}^{\left[\frac{\f (sh)}{h}\right] + n_0}G_j(\a) = \sum_{j=1}^{n_0}G_j(\a)\sum_{|s| \leq s(n_0-j)} G_s(\a)
%\]
For every $j = 1, \ldots, n_0$, by (\ref{expr}), we have that
\[
\sum_{|s| \leq s(n_0-j)} G_s(\a) = 1 - \sum_{|s|>s(n_0-j)}G_s(\a) = 1 - \sum_{|s|>\sqrt{\frac{2(n_0 - j)}{\kappa h}}} G_s(\a)
\]
Similarly to the proof of Theorem \ref{P1.Case1}, we perform the following re-scaling, using the fact that $t = \mu h$:
\begin{equation}\label{sum1}
\sum_{|s|>\sqrt{\frac{n_0 - j}{\kappa h}}} G_s(\a) = \int_{|s|>\sqrt{\frac{2(n_0 - j)}{\kappa h}}} G_{[s]}(\a) d s = 2 \frac{\sqrt{\t}}{h} \int_{\sqrt{\frac{2(n_0 - j)}{\mu \kappa}}}^{\infty} G_{[\frac{\sqrt{\t} x}{h}]}(\a) dx
\end{equation}
Using Proposition 1, we can write the above as
\begin{eqnarray}\label{sum2}
\frac{\sqrt{\t}}{h} \int_{\sqrt{\frac{2(n_0 - j)}{\mu \kappa}}}^{\infty} G_{[\frac{\sqrt{\t} x}{h}]}(\a) dx = \frac{1}{\sqrt{4 \pi}} RS_{\frac{h}{\sqrt{\t}}}\left[e^{-\frac{x^2}{4}}, \left(\sqrt{\frac{2(n_0 - j)}{\mu \kappa}}, +\infty \right)\right] + \nonumber\\
+ C \frac{h}{\sqrt{\t}}RS_{\frac{h}{\sqrt{\t}}}\left[\frac{1}{x}e^{-\frac{x^2}{4}}, \left(\sqrt{\frac{2(n_0 - j)}{\mu \kappa}}, +\infty \right)\right].
\end{eqnarray}
Therefore, using the uniform Riemann sum estimate (\ref{RSest1D}), we have
\begin{equation*}\label{Discr_to_cont}
\sum_{|s|>\sqrt{\frac{n_0 - j}{\kappa h}}} G_s(\a) = \frac{2}{\sqrt{4\pi}} \int_{\sqrt{\frac{2(n_0 - j)}{\mu \kappa}}}^{\infty} e^{-\frac{x^2}{4}} dx + O(\sqrt{h})
\end{equation*}
By property (\ref{property1}),
\[
\sum_{|s| \leq s(n_0-j)} G_s(\a) = 1 - \sum_{|s|>\sqrt{\frac{n_0 - j}{\kappa h}}} G_s(\a)  =  \frac{2}{\sqrt{4\pi}}\int_{0}^{\sqrt{\frac{2(n_0 - j)}{\mu \kappa}}} e^{-\frac{x^2}{4}} dx + O(\sqrt{h})
\]
Combining the latter expansion with (\ref{asymp1}), we have
\begin{equation}\label{eq1}
\sum_{j=1}^{n_0}G_j(\a)\sum_{|s| \leq s(n_0-j)} G_s(\a) = 2 \frac{\sqrt{h}}{4 \pi \sqrt{\mu}} \sum_{k = 1}^{n_0 - 1} \int_{0}^{\sqrt{\frac{2 k}{\mu \kappa}}} e^{- \frac{x^2}{4}} dx + O(h), h \to 0.
\end{equation}
Similarly,
\begin{equation}\label{eq2}
\frac{1}{2}G_0(\a)\sum_{|s| \leq S_0}G_s(\a) = \frac{\sqrt{h}}{4 \pi \sqrt{\mu}} \int_{0}^{\sqrt{\frac{2 n_0}{\mu \kappa}}} e^{- \frac{x^2}{4}} dx + O(h), h \to 0,
\end{equation}
\begin{equation}\label{eq3}
\frac{1}{2}G_0(\a)\sum_{|s| \geq S_0+1}G_s(\a) = \frac{\sqrt{h}}{4 \pi \sqrt{\mu}} \int_{\sqrt{\frac{2 n_0}{\mu \kappa}}}^{\infty} e^{-\frac{x^2}{4}} dx + O(h), h \to 0.
\end{equation}
and
\begin{equation}\label{eq4}
\sum_{j = -N}^{-1} G_j(\a) \sum_{|s| > s(n_0 - j + 1)} G_s(\a) = 2 \frac{\sqrt{h}}{4 \pi \sqrt{\mu}} \sum_{k = n_0 + 1}^{n_0 + N} \int_{\sqrt{\frac{2 k}{\mu \kappa}}}^{\infty} e^{-\frac{x^2}{4}} dx + O(h)
\end{equation}
Combining (\ref{eq1}), (\ref{eq2}), (\ref{eq3}) and (\ref{eq4}), the inequality (\ref{identityLB}) implies
\[
\sum_{k = 1}^{n_0 - 1} \int_{0}^{\sqrt{\frac{2k}{\mu \kappa}}} e^{- \frac{x^2}{4}} dx + \frac{1}{2} \int_{0}^{\sqrt{\frac{2 n_0}{\mu \kappa}}} e^{- \frac{x^2}{4}} dx \geq \frac{1}{2} \int_{\sqrt{\frac{2 n_0}{\mu \kappa}}}^{\infty} e^{- \frac{x^2}{4}} dx + \sum_{k = n_0 + 1}^{n_0 + N} \int_{\sqrt{\frac{2 k}{\mu \kappa}}}^{\infty} e^{- \frac{x^2}{4}} dx
\]
and since $N \geq 1$ was chosen arbitrarily, letting $N\Converge\infty$,
we arrive at:
\begin{equation}\label{lowerB}
\sum_{k = 1}^{n_0 - 1} \int_{0}^{\sqrt{\frac{2k}{\mu \kappa}}} e^{- \frac{x^2}{4}} dx + \frac{1}{2} \int_{0}^{\sqrt{\frac{2 n_0}{\mu \kappa}}} e^{- \frac{x^2}{4}} dx \geq \frac{1}{2} \int_{\sqrt{\frac{2 n_0}{\mu \kappa}}}^{\infty} e^{- \frac{x^2}{4}} dx + \sum_{k = n_0 + 1}^{\infty} \int_{\sqrt{\frac{2k}{\mu \kappa}}}^{\infty} e^{- \frac{x^2}{4}} dx
\end{equation}
Thus, (i) follows.

\subsection{The matching upper bound.}
%\colorRED{Some explanation to highlight the differences and similarities
%between the lower bound.}

Now, let $w^{0,-n_0}(\t) \leq \frac{1}{2}$. It follows from (\ref{asymp1}), as well as from the monotonicity of $G_n(\a)$ in $n$, that the inequality
\[
G_n(\a) \leq \sqrt{\frac{h}{4 \pi \mu}}
\]
holds for all $n \geq 1$. Thus, from (\ref{identity1}) we have
\begin{multline}\label{identityUB}
\sum_{j=1}^{n_0}G_j(\a)\sum_{|s| \leq s(n_0-j)} G_s(\a) + \frac{1}{2}G_0(\a)\sum_{|s| \leq s(-n_0)}G_s(\a) \leq \\
\leq \frac{1}{2}G_0(\a)\sum_{|s| \geq s(-n_0)}G_s(\a) +  \sqrt{\frac{h}{4 \pi}} \sum_{j = -\infty}^{-1} \sum_{|s| > s(n_0 - j + 1)} G_s(\a)
\end{multline}
Now, fix $j \leq -1$, and denote $m := n_0 - j + 1$. Using (\ref{sum1}) and (\ref{sum2}), we have
\begin{eqnarray}\label{p}
\sum_{|s| > s(m)} G_s(\a) = \frac{2}{\sqrt{4 \pi}} RS_{\frac{h}{\sqrt{\t}}}\left[e^{-\frac{x^2}{4}}, \left(\sqrt{\frac{2 m}{\mu \kappa}}, +\infty \right)\right] + \\
\nonumber + C \frac{h}{\sqrt{\t}}RS_{\frac{h}{\sqrt{\t}}}\left[\frac{1}{x}e^{-\frac{x^2}{4}}, \left(\sqrt{\frac{2 m}{\mu \kappa}}, +\infty \right)\right].
\end{eqnarray}
Note that
\begin{equation}\label{q1}
\left|RS_{\frac{h}{\sqrt{\t}}}\left[e^{-\frac{x^2}{4}}, \left(\sqrt{\frac{2 m}{\mu \kappa}}, +\infty \right)\right] - \int_{\sqrt{\frac{2m}{\mu \kappa}}}^{\infty} e^{- \frac{x^2}{4}} dx \right| \leq \frac{h}{\sqrt{\t}} e^{-\frac{m}{2 \mu \kappa}}
\end{equation}
and
\begin{equation}\label{q2}
\left|RS_{\frac{h}{\sqrt{\t}}}\left[\frac{1}{x}e^{-\frac{x^2}{4}}, \left(\sqrt{\frac{2 m}{\mu \kappa}}, +\infty \right)\right] - \int_{\sqrt{\frac{2m}{\mu \kappa}}}^{\infty} \frac{1}{x}e^{- \frac{x^2}{4}} dx \right| \leq \frac{h}{\sqrt{\t}} e^{-\frac{m}{2 \mu \kappa}}.
\end{equation}
Moreover, we have
\begin{equation}\label{q3}
\int_{\sqrt{\frac{2m}{\mu \kappa}}}^{\infty} \frac{1}{x} e^{-\frac{x^2}{4}} dx \leq C \int_{\sqrt{\frac{2m}{\mu \kappa}}}^{\infty} \frac{x}{2} e^{-\frac{x^2}{4}} dx = C e^{-\frac{m}{2 \mu \kappa}}
\end{equation}
Combining (\ref{q1}), (\ref{q2}) and (\ref{q3}), the equation (\ref{p}) reads as
\[
\sum_{|s| > s(m)} G_s(\a) = \frac{2}{\sqrt{4 \pi}} \int_{\sqrt{\frac{2m}{\mu \kappa}}}^{\infty} e^{- \frac{x^2}{4}} dx + \frac{2}{\sqrt{4 \pi \mu}}  e^{-\frac{m}{2 \mu \kappa}}\left(\sqrt{h} + o(\sqrt{h})\right)
\]
Therefore,

%where $R[f(x), [a,b], \Delta x]$ denotes the error between $\int_{a}^{b} f(x) dx$ and the Riemann sum for $f$ over this interval with step $\Delta x$. Since
%\[
%\sum_{m=n_0 + 1}^{\infty} \int_{\sqrt{\frac{m}{\kappa}}} C_{h,x} \frac{1}{x} e^{-\frac{x^2}{4}} dx < \infty
%\]
%we have
%\begin{equation}\label{ERR1}
%\sum_{m = n_0 + 1}^{\infty} \frac{1}{\sqrt{4\pi}} \int_{\sqrt{\frac{m}{\kappa}}} C_{h,x} \frac{\sqrt{h}}{x} e^{-\frac{x^2}{4}} dx = O(\sqrt{h})
%\end{equation}
%Furthermore,
%\[
%R[e^{-\frac{x^2}{4}},\left[\sqrt{\frac{m}{\kappa}}, \infty\right),\sqrt{h}] \leq e^{- \frac{m}{4 \kappa}} \sqrt{h}
%\]
%hence
%\begin{equation}\label{ERR2}
%\sum_{m = n_0 + 1}^{\infty} R[e^{-\frac{x^2}{4}},\left[\sqrt{\frac{m}{\kappa}}, \infty\right),\sqrt{h}] = O(\sqrt{h})
%\end{equation}

\[
\sum_{j = -\infty}^{-1} \sum_{|s| > s(n_0 - j + 1)} G_s(\a) = \frac{1}{\sqrt{4\pi}}\sum_{k = n_0 + 1}^{\infty} \int_{\sqrt{\frac{2k}{\mu\kappa}}}^{\infty} e^{- \frac{x^2}{4}} dx + O(\sqrt{h}).
\]
Using (\ref{eq1}), (\ref{eq2}) and (\ref{eq3}), the inequality (\ref{identityUB}) yields the upper bound
\begin{equation}\label{upperB}
\sum_{k = 1}^{n_0 - 1} \int_{0}^{\sqrt{\frac{2 k}{\mu \kappa}}} e^{- \frac{x^2}{4}} dx + \frac{1}{2} \int_{0}^{\sqrt{\frac{2 n_0}{\mu \kappa}}} e^{- \frac{x^2}{4}} dx \leq \frac{1}{2} \int_{\sqrt{\frac{2 n_0}{\mu \kappa}}}^{\infty} e^{- \frac{x^2}{4}} dx + \sum_{k = n_0 + 1}^{\infty} \int_{\sqrt{\frac{2k}{\mu \kappa}}}^{\infty} e^{- \frac{x^2}{4}} dx
\end{equation}
which shows (ii).

\subsection{Proof of Theorem \ref{Th3}:
the limiting case $\mu \to \infty$.}
Recall that $n_0$ satisfies
\Beqn\label{I1}
\sum_{k = 1}^{n_0 - 1} \int_{0}^{\sqrt{\frac{2k}{\mu \kappa}}} e^{- \frac{x^2}{4}} dx + \frac{1}{2} \int_{0}^{\sqrt{\frac{2 n_0}{\mu \kappa}}} e^{- \frac{x^2}{4}} dx \leq \frac{1}{2} \int_{\sqrt{\frac{2 n_0}{\mu \kappa}}}^{\infty} e^{- \frac{x^2}{4}} dx + \sum_{k = n_0 + 1}^{\infty} \int_{ \sqrt{\frac{2 k}{\mu \kappa}}}^{\infty} e^{- \frac{x^2}{4}} dx.
\Eeqn
and
\Beqn\label{I2}
\sum_{k = 1}^{n_0} \int_{0}^{\sqrt{\frac{2k}{\mu \kappa}}} e^{- \frac{x^2}{4}} dx + \frac{1}{2} \int_{0}^{\sqrt{\frac{2 n_0 + 2}{\mu \kappa}}} e^{- \frac{x^2}{4}} dx \geq \frac{1}{2} \int_{\sqrt{\frac{2 n_0 + 2}{\mu \kappa}}}^{\infty} e^{- \frac{x^2}{4}} dx + \sum_{k = n_0 + 2}^{\infty} \int_{ \sqrt{\frac{2 k}{\mu \kappa}}}^{\infty} e^{- \frac{x^2}{4}} dx.
\Eeqn
We first start with the following claim:
as $\mu \to \infty$, up to a subsequence, we have
\begin{equation}\label{claim}
\lim_{\mu \to \infty}\frac{n_0}{ \mu \kappa} = a, \text{ with } 0 < a < \infty.
\end{equation}

First, suppose $a = \infty$, i.e. $n_0 \gg\mu$.
Then, for sufficiently large $\mu$, we have $\frac{2k}{\mu \kappa} \geq 1$ for all $k \geq n_0 + 1$. Hence, we can estimate the right hand side of (\ref{I1})
in the following way:
\[
\sum_{k = n_0+1}^{\infty} \int_{\sqrt{\frac{2k}{\mu \kappa}}}^{\infty} e^{- \frac{x^2}{4}} dx \leq
\sum_{k = n_0+1}^{\infty} \int_{\sqrt{\frac{2k}{\mu \kappa}}}^{\infty} x e^{- \frac{x^2}{4}} dx  = 2 \frac{e^{\frac{-2(n_0 + 1)}{\mu \kappa}}}{1 - e^{-\frac{1}{2 \mu \kappa}}} = 2e^{\frac{-2(n_0 + 1)}{ \mu \kappa}} (2 \mu \kappa + o(\mu)) = o(\mu), \mu \to \infty.
\]
On the other hand, for sufficiently large $\mu$, we arrive at contradiction with (\ref{I1}), since
\[
\sum_{k = 1}^{n_0 - 1} \int_{0}^{\sqrt{\frac{2k}{\mu \kappa}}} e^{- \frac{x^2}{4}} dx \geq \sum_{k = 1}^{\mu \kappa + 1} \int_{0}^{\sqrt{\frac{2k}{\mu \kappa}}} e^{- \frac{x^2}{4}} dx \geq \sum_{k = 1}^{\mu \kappa + 1} \int_{0}^{\sqrt{\frac{2 k}{\mu \kappa}}} x e^{- \frac{x^2}{4}} dx = \frac{4}{e} \mu \kappa + o(\mu), \mu \to \infty
\]

Second, assume $a = 0$, i.e. $n_0\ll\mu$. In this case, the contradiction with (\ref{I2}) follows from the fact that
\[
\sum_{k = 1}^{n_0} \int_{0}^{\sqrt{\frac{2 k}{\mu \kappa}}} e^{- \frac{x^2}{4}} dx  \leq n_0 \int_{0}^{\sqrt{\frac{2 n_0}{\mu \kappa}}} e^{- \frac{x^2}{4}} dx = o(n_0)
\]
while
\[
\sum_{k = n_0+2}^{\infty} \int_{\sqrt{\frac{2 k}{\mu \kappa}}}^{\infty} e^{- \frac{x^2}{4}} dx \geq \sum_{k = n_0+2}^{2 n_0 + 1} \int_{\sqrt{\frac{2 k}{\mu \kappa}}}^{\infty} e^{- \frac{x^2}{4}} dx  \geq n_0 \int_{1}^{\infty} e^{- \frac{x^2}{4}} dx,
\]
thus the claim (\ref{claim}) holds.

It remains to show that $a = 1$. Suppose $a>1$. Omitting the integer part notation, assume that $n_0 = a \mu \kappa$, then
\[
\sum_{k = n_0+1}^{\infty} \int_{\sqrt{\frac{2 k}{\mu \kappa}}}^{\infty} e^{- \frac{x^2}{4}} dx  = \sum_{k = \mu \kappa}^{\infty} \int_{\sqrt{\frac{2 k}{\mu \kappa}}}^{\infty} e^{- \frac{x^2}{4}} dx - \sum_{k = \mu \kappa + 1}^{a \mu \kappa + 1} \int_{\sqrt{\frac{2 k}{\mu \kappa}}}^{\infty} e^{- \frac{x^2}{4}}
\]
and
\[
\sum_{k = 1}^{n_0 - 1} \int_{0}^{\sqrt{\frac{2 k}{\mu \kappa}}} e^{- \frac{x^2}{4}} dx = \sum_{k = 1}^{\mu \kappa}  \int_{0}^{\sqrt{\frac{2 k}{\mu \kappa}}} e^{- \frac{x^2}{4}} dx + \sum_{k = \mu \kappa + 1}^{a \mu \kappa - 1}  \int_{0}^{\sqrt{\frac{2 k}{\mu \kappa}}} e^{- \frac{x^2}{4}} dx
\]
This way, the equation (\ref{I1}) reads as
\[
\sum_{k = 1}^{\mu \kappa}  \int_{0}^{\sqrt{\frac{2 k}{\mu \kappa}}} e^{- \frac{x^2}{4}} dx + \sum_{k = \mu \kappa + 1}^{a \mu \kappa - 1}  \int_{0}^{\sqrt{\frac{2 k}{\mu \kappa}}} e^{- \frac{x^2}{4}} dx + \sum_{k = \mu \kappa + 1}^{a \mu \kappa + 1} \int_{\sqrt{\frac{2 k}{\mu \kappa}}}^{\infty} e^{- \frac{x^2}{4}} \leq \sum_{k = \mu \kappa}^{\infty} \int_{\sqrt{\frac{2 k}{\mu \kappa}}}^{\infty} e^{- \frac{x^2}{4}} dx + O(1), \mu \to \infty,
\]
or
\begin{equation}\label{otnosh}
\sum_{k = 1}^{\mu \kappa}  \int_{0}^{\sqrt{\frac{2 k}{\mu \kappa}}} e^{- \frac{x^2}{4}} dx + 2 \pi \mu \kappa (a-1) \leq  \sum_{k = \mu \kappa}^{\infty} \int_{\sqrt{\frac{2 k}{\mu \kappa}}}^{\infty} e^{- \frac{x^2}{4}} dx + O(1), \mu \to \infty.
\end{equation}
Thus, in order to obtain contradiction, it suffices to show that

%The identity (\ref{otnosh}) yields a contradiction since
\Beqn\label{equality}
\lim_{n \to \infty} \frac{1}{n} \sum_{k = 1}^{n}  \int_{0}^{\sqrt{\frac{2 k}{n}}} e^{- \frac{x^2}{4}} dx = \lim_{n \to \infty} \frac{1}{n} \sum_{k = n}^{\infty} \int_{\sqrt{\frac{2 k}{n}}}^{\infty} e^{- \frac{x^2}{4}} dx.
%\approx 0.857764.
\Eeqn

%\[
%\Psi(x):= \int_{0}^{\sqrt{2 x}}  e^{- \frac{y^2}{4}} dy.
%\]
Note that
\begin{multline*}
\frac{1}{n} \sum_{k = 1}^{n}  \int_{0}^{\sqrt{\frac{2 k}{n}}} e^{- \frac{x^2}{4}} dx = RS_{\frac{1}{n}}\left[\int_{0}^{\sqrt{2 x}}  e^{- \frac{y^2}{4}} dy, [0,1]\right]
= \int_{0}^{1} \int_{0}^{\sqrt{2 x}}  e^{- \frac{y^2}{4}} dy dx + o_n(1), \ n \to \infty,
\end{multline*}
and
\begin{multline*}
\frac{1}{n} \sum_{k = n}^{\infty} \int_{\sqrt{\frac{2 k}{n}}}^{\infty} e^{- \frac{x^2}{4}} dx = RS_{\frac{1}{n}}\left[\int_{\sqrt{2 x}}^{\infty}  e^{- \frac{y^2}{4}} dy, [1,\infty)\right]
= \int_{1}^{\infty} \int_{\sqrt{2 x}}^{\infty}  e^{- \frac{y^2}{4}} dy dx + o_n(1),  \ n \to \infty,
\end{multline*}
To establish (\ref{equality}), it remains to show that
\Beqn\label{equality1}
\int_{0}^{1} \int_{0}^{\sqrt{2 x}}  e^{- \frac{y^2}{4}} dy dx = \int_{1}^{\infty} \int_{\sqrt{2 x}}^{\infty}  e^{- \frac{y^2}{4}} dy dx.
\Eeqn
Since $\int_{0}^{\infty}  e^{- \frac{y^2}{4}} dy = \sqrt{\pi}$, we have
\[
\int_{0}^{1} \int_{0}^{\sqrt{2 x}}  e^{- \frac{y^2}{4}} dy dx = \sqrt{\pi} - \int_{0}^{1} \int_{\sqrt{2 x}}^{\infty} e^{- \frac{y^2}{4}} dy dx
\]
thus (\ref{equality1}) is equivalent to
\Beqn\label{equality2}
\int_{0}^{\infty} \int_{\sqrt{2 x}}^{\infty}  e^{- \frac{y^2}{4}} dy dx = \sqrt{\pi}
\Eeqn
The latter identity is true since
\[
\int_{0}^{\infty} \int_{\sqrt{2 x}}^{\infty}  e^{- \frac{y^2}{4}} dy dx = \int_{0}^{\infty} \int_{0}^{\frac{y^2}{2}}  e^{- \frac{y^2}{4}} dx dy =  \int_{0}^{\infty} \frac{y^2}{2}  e^{- \frac{y^2}{4}} dy = \int_{0}^{\infty}  e^{- \frac{y^2}{4}} dy = \sqrt{\pi}
\]
Consequently, (\ref{equality}) holds, which yields a contradiction with (\ref{otnosh}) if $a>1$. The case $a<1$ is excluded analogously.

\begin{proof}[Proof of corollary 1.]
The proof follows by observing that when $\mu\kappa$ is small enough,
the left hand side of \eqref{eq_key} dominates the right hand side even for $n_0=1$.
Precisely, if $\mu \kappa \to 0$ and $n_0=1$,
\[
\int_{0}^{\sqrt{\frac{2}{\mu \kappa}}} e^{-\frac{x^2}{4}} dx \to \sqrt{\pi}
\]
while
\[
\sum_{k = 1}^{\infty} \int_{\sqrt{\frac{2k}{\mu \kappa}}}^{\infty} e^{- \frac{x^2}{4}} dx \leq
\sum_{k = 1}^{\infty} \int_{\sqrt{\frac{2k}{\mu \kappa}}}^{\infty} x e^{- \frac{x^2}{4}} dx  = \frac{2 e^{-\frac{2}{\mu \kappa}}}{1 - e^{-\frac{1}{2 \mu \kappa}}} \to 0, \mu \kappa \to 0.
\]
\end{proof}

\subsection{Anisotropic curvature-driven motions.} This subsection is devoted to the proof of Theorem \ref{Th4}. We are going to present the proof assuming $w_{s_{n_0}, j_{n_0}}(\t) = \frac{1}{2}$. The corresponding inequalities if $w_{s_{n_0}, j_{n_0}}(\t) \leq \frac{1}{2}$ or $w_{s_{n_0}, j_{n_0}}(\t) \geq \frac{1}{2}$ can be obtained analogously. As before, we start with noting that
\begin{equation}\label{angle1}
\frac{1}{2} = w_{s_{n_0}, j_{n_0}}(\t) = \sum_{(s^{'} h,j^{'} h) \in \Omega} G_{s^{'} - s_{n_0}}(\a) G_{j^{'} - j_{n_0}}(\a) = \sum_{(s h + s_{n_0} h, j h + j_{n_0} h) \in \Omega} G_{s}(\a) G_{j}(\a)
\end{equation}
Introduce the following sets:
\[
I_1^{\theta}:=\{(s,j): \frac{p}{q} sh < jh < g(sh - s_{n_0} h) - j_{n_0} h\};
\]
\[
I_2^{\theta}:=\{(s,j):g(sh - s_{n_0} h) - j_{n_0} h < jh < \frac{p}{q} s h\}
\]
and
\[
I_3^{\theta}:=\{(s,j): j h < \min \{g(sh - s_{n_0} h) - j_{n_0} h, \frac{p}{q} s h\}\}.
\]
Also, denote
\[
bI_1^{\theta}:=\{(s,j): s = \frac{p}{q} j, jh < g(sh - s_{n_0} h) - j_{n_0} h\}
\]
and
\[
bI_2^{\theta}:=\{(s,j): s = \frac{p}{q} j, jh \geq g(sh - s_{n_0} h) - j_{n_0} h\}
\]
With this notation, (\ref{angle1}) reads as
\[
\sum_{(s,j) \in I_1^{\theta} \cup I_3^{\theta}} G_{s}(\a) G_{j}(\a) + \sum_{(s,j) \in bI_1^{\theta}} G_{s}(\a) G_{j}(\a) = \frac{1}{2}
\]
On the other hand, the properties (\ref{property1}) and (\ref{property2}) of the heat kernel yield that
\[
\sum_{(s,j) \in I_2^{\theta} \cup I_3^{\theta}} G_{s}(\a) G_{j}(\a) + \frac{1}{2} \sum_{(s,j) \in bI_1^{\theta} \cup bI_2^{\theta}} G_{s}(\a) G_{j}(\a) = \frac{1}{2}
\]
thus
\begin{equation}\label{angle_identity}
\sum_{(s,j) \in I_1^{\theta}} G_{s}(\a) G_{j}(\a) + \frac{1}{2} \sum_{(s,j) \in bI_1^{\theta}} G_{s}(\a) G_{j}(\a) = \frac{1}{2} \sum_{(s,j) \in bI_2^{\theta}} G_{s}(\a) G_{j}(\a) + \sum_{(s,j) \in I_2^{\theta}} G_{s}(\a) G_{j}(\a)
\end{equation}
A particular technical difficulty, which is associated with the analysis of (\ref{angle_identity}), stems in the fact that, for given $h>0$, an arbitrary line, whose tangent is different from $\tan \theta$, intersects $I_1^{\theta}$ at most at a bounded (i.e. $O(1)$) number of points. Therefore, if we choose a natural summation order along the vertical and horizontal grid lines, when describing the sum $\sum_{(s,j) \in I_1^{\theta}} G_{s}(\a) G_{j}(\a)$, we cannot apply the asymptotic analysis described above and used in the proof of Theorem \ref{Th2}, i.e. (\ref{expansion_key}). However, there are finitely many lines at angle $\theta$, which intersect $I_1^{\theta}$ at $O\left(\frac{1}{\sqrt{h}}\right)$ number of points. Specifically, those are the lines  at angle $\theta$ that pass through the points $(s_i,j_i)$ for $1 \leq i \leq n_0-1$. Therefore, we will have to perform the asymptotic analysis only along those lines. We proceed with asymptotic expansions of the terms in (\ref{angle_identity}):
\begin{enumerate}
\item
\[
\sum_{(s,j) \in bI_1^{\theta}} G_{s}(\a) G_{j}(\a) = \sum_{s=s_{n_0}^{-}}^{s_{n_0}^{+}} G_{s q}(\a) G_{s p}(\a)
\]
where, omitting the integer part notation
\[
s_{n_0}^{\pm} = \pm \sqrt{\frac{2}{h \kappa}} \frac{\sqrt{p s_{n_0} - q j_{n_0}}}{(p^2+q^2)^{3/4}} - \frac{s_{n_0}}{q}
\]
We next perform the same rescaling as in (\ref{sum1}):
\[
\sum_{s=s_{n_0}^{-}}^{s_{n_0}^{+}} G_{s q}(\a) G_{s p}(\a) = \int_{s_{n_0}^{-}}^{s_{n_0}^{+}} G_{[s] q}(\a) G_{[s] p}(\a) ds = \frac{\sqrt{\t}}{h}\int_{\frac{h}{\sqrt{\t}} s_{n_0}^{-}}^{\frac{h}{\sqrt{\t}} s_{n_0}^{+}} G_{\left[ \frac{\sqrt{\t}}{h} x\right] q}(\a) G_{\left[\frac{\sqrt{\t}}{h} x \right] p }(\a) dx
\]
Note that for $t = \mu h$
\[
s_{n_0}^{\pm} \frac{h}{\sqrt{\t}} = \pm \sqrt{\frac{2}{\mu \kappa}} \frac{\sqrt{p s_{n_0} - q j_{n_0}}}{(p^2+q^2)^{3/4}} + \frac{s_{n_0} \sqrt{h}}{q \sqrt{\mu}} = \pm \sqrt{\frac{2 d_{n_0}}{\mu \kappa (p^2+q^2)}} + \frac{s_{n_0} \sqrt{h}}{q \sqrt{\mu}}.
\]
Thus, using the same reasoning as in the proof of Theorem \ref{Th2}, (\ref{sum2}), we have
\begin{eqnarray}\label{angle_exp1}
\sum_{s=s_{n_0}^{-}}^{s_{n_0}^{+}} G_{s q}(\a) G_{s p}(\a) = \frac{2}{4 \pi} \frac{\sqrt{h}}{\sqrt{\mu}}\int_{0}^{\sqrt{\frac{2 d_{n_0}}{\mu \kappa (p^2+q^2)}}} e^{-\frac{(p^2 + q^2) x^2}{4}} dx + O(h \ln h) = \\
\nonumber = \frac{2}{4 \pi} \frac{\sqrt{h}}{\sqrt{\mu}} \frac{1}{\sqrt{p^2+q^2}} \int_{0}^{\sqrt{\frac{2 d_{n_0}}{\mu \kappa}}} e^{-\frac{x^2}{4}} dx + O(h \ln h)
\end{eqnarray}
\item
Similarly,
\begin{eqnarray}\label{angle_exp2}
  \sum_{(s,j) \in bI_2^{\theta}} G_{s}(\a) G_{j}(\a) = \sum_{s = -\infty}^{s_{n_0}^{-} -1} G_{s q}(\a) G_{s p}(\a) + \sum_{s = s_{n_0}^{+} +1}^{\infty} G_{s q}(\a) G_{s p}(\a) = \\
  \nonumber =\frac{2}{4 \pi} \frac{\sqrt{h}}{\sqrt{\mu}} \frac{1}{\sqrt{p^2+q^2}} \int_{\sqrt{\frac{2 d_{n_0}}{\mu \kappa}}}^{\infty} e^{-\frac{x^2}{4}} dx + O(h)
\end{eqnarray}
\item
We have
\[
\sum_{(s,j) \in I_1^{\theta}} G_{s}(\a) G_{j}(\a) = \sum_{i=1}^{n_0-1} \sum_{s=s_i^{-}}^{s_i^{+}} G_{s_i-s_{n_0} + s q}(\alpha)G_{j_i-j_{n_0} + s p}(\alpha)
\]
where
\[
s_{i}^{\pm} = \pm \sqrt{\frac{2}{h \kappa}} \frac{\sqrt{p s_{i} - q j_{i}}}{(p^2+q^2)^{3/4}} - \frac{s_{i}}{q}.
\]

%\begin{figure}[!h]
%%\begin{center}
%\centering\fig{figure=picture1.pdf,height=0.5cm}
%\end{figure}

Proceeding as in item 1, every $i$ we have
\[
\sum_{s=s_i^{-}}^{s_i^{+}} G_{s_i-s_{n_0} + s q}(\alpha)G_{j_i-j_{n_0} + s p}(\alpha) = \frac{\sqrt{\t}}{h}\int_{\frac{h}{\sqrt{\t}} s_{i}^{-}}^{\frac{h}{\sqrt{\t}} s_{i}^{+}} G_{\left[ \frac{\sqrt{\t}}{h} x\right] q + s_i-s_{n_0}}(\a) G_{\left[\frac{\sqrt{\t}}{h} x \right] p + j_i-j_{n_0}}(\a) dx.
\]
Since for $\t = \mu h$ we have $\frac{\sqrt{\t}}{h} x p + j_i-j_{n_0} = \frac{\sqrt{\t}}{h} \left(x p - \frac{\sqrt{h}}{\sqrt{\mu}}( j_i-j_{n_0})\right)$,  $\frac{\sqrt{\t}}{h} x q + s_i-s_{n_0} = \frac{\sqrt{\t}}{h} \left(x q - \frac{\sqrt{h}}{\sqrt{\mu}}( s_i-s_{n_0})\right)$, and
\[
s_{i}^{\pm} \frac{h}{\sqrt{\t}} = \pm \sqrt{\frac{2}{\mu \kappa}} \frac{\sqrt{p s_{i} - q j_{i}}}{(p^2+q^2)^{3/4}} + \frac{s_{i} \sqrt{h}}{q \sqrt{\mu}} = \pm \sqrt{\frac{2 d_{i}}{\mu \kappa (p^2+q^2)}} + \frac{s_{i} \sqrt{h}}{q \sqrt{\mu}}.
\]
we have
\begin{multline}\label{angle_exp3_prel}
\sum_{s=s_{i}^{-}}^{s_{i}^{+}}G_{s_i-s_{n_0} + s q}(\alpha)G_{j_i-j_{n_0} + s p}(\alpha) = \\
\frac{2}{4 \pi} \frac{\sqrt{h}}{\sqrt{\mu}} \int_{0}^{\sqrt{\frac{2 d_{i}}{\mu \kappa (p^2+q^2)}}} e^{-\frac{p^2 \left(x - \frac{\sqrt{h}}{\sqrt{\mu}}( j_i-j_{n_0}) \right)^2}{4}}  e^{-\frac{q^2\left(x - \frac{\sqrt{h}}{\sqrt{\mu}}( s_i-s_{n_0})\right)^2}{4}} dx
+ O(h \ln h)
\end{multline}
Finally, using the expansion $e^{-(x+ a\sqrt{h})^2} = e^{-x^2} - 2 a x \sqrt{h} e^{-x^2} + o(\sqrt{h})$, and changing variables $y \to \sqrt{p^2+q^2} x$, (\ref{angle_exp3_prel}) yields
\begin{eqnarray}\label{angle_exp3}
\sum_{s=s_{i}^{-}}^{s_{i}^{+}}G_{s_i-s_{n_0} + s q}(\alpha)G_{j_i-j_{n_0} + s p}(\alpha)  =
\frac{2}{4 \pi} \frac{\sqrt{h}}{\sqrt{\mu}} \frac{1}{\sqrt{p^2+q^2}} \int_{0}^{\sqrt{\frac{2 d_{i}}{\mu \kappa}}} e^{-\frac{x^2}{4}} dx.
\end{eqnarray}
Therefore
\begin{equation}\label{angle_exp3F}
\sum_{(s,j) \in I_1^{\theta}} G_{s}(\a) G_{j}(\a) = \frac{2}{4 \pi} \frac{\sqrt{h}}{\sqrt{\mu}} \frac{1}{\sqrt{p^2+q^2}} \sum_{i=1}^{n_0-1} \int_{0}^{\sqrt{\frac{2 d_{i}}{\mu \kappa}}} e^{-\frac{x^2}{4}} dx + o(\sqrt{h})
\end{equation}
\item For the last term in (\ref{angle_identity}), we may use the same notation and reasoning as in part 3. Note, that if $p \neq 0$, the asymptotic expansion of this term can be obtained somewhat easier then the one in Theorem 2 (where $p = 0$) since there is no need to establish the upper and lower bounds in this case. We have
\begin{eqnarray}\label{angle_exp4}
\sum_{(s,j) \in I_2^{\theta}} G_{s}(\a) G_{j}(\a) = 2 \sum_{i = n_0 + 1}^{\infty} \sum_{s = s_{i}^{+}}^{\infty} G_{s_i-s_{n_0} + s q}(\a) G_{j_i-j_{n_0} + s p}(\a) = \\
\nonumber = \frac{2}{4 \pi} \frac{\sqrt{h}}{\sqrt{\mu}} \frac{1}{\sqrt{p^2+q^2}} \sum_{i=n_0+1}^{\infty} \int_{\sqrt{\frac{2 d_{i}}{\mu \kappa}}}^{\infty} e^{-\frac{x^2}{4}} dx + o(\sqrt{h}).
\end{eqnarray}
\end{enumerate}
This completes the proof of Theorem \ref{Th4}.
\vspace{-0.5cm}
\section{Numerical experiments}
In this section we implemented the scheme (\ref{semidiscrete}) numerically for various initial data $\Omega_0$ and parameters $t$ and $h$. The simulation is performed using Matlab and Fast Fourier Transform.

We start with observing the evolution of a circle and of a dumpbell-shaped domain. In the following figures, the blue contour indicates the initial configuration, the red contour indicates the final configuration after 10 steps. Notice the evolution picture for the circle does not have the red contour since the circle vanishes in less then 10 steps:

\vspace{10pt}

\begin{minipage}[c]{2.5in}
\begin{center}
\includegraphics [width = 2.8in]{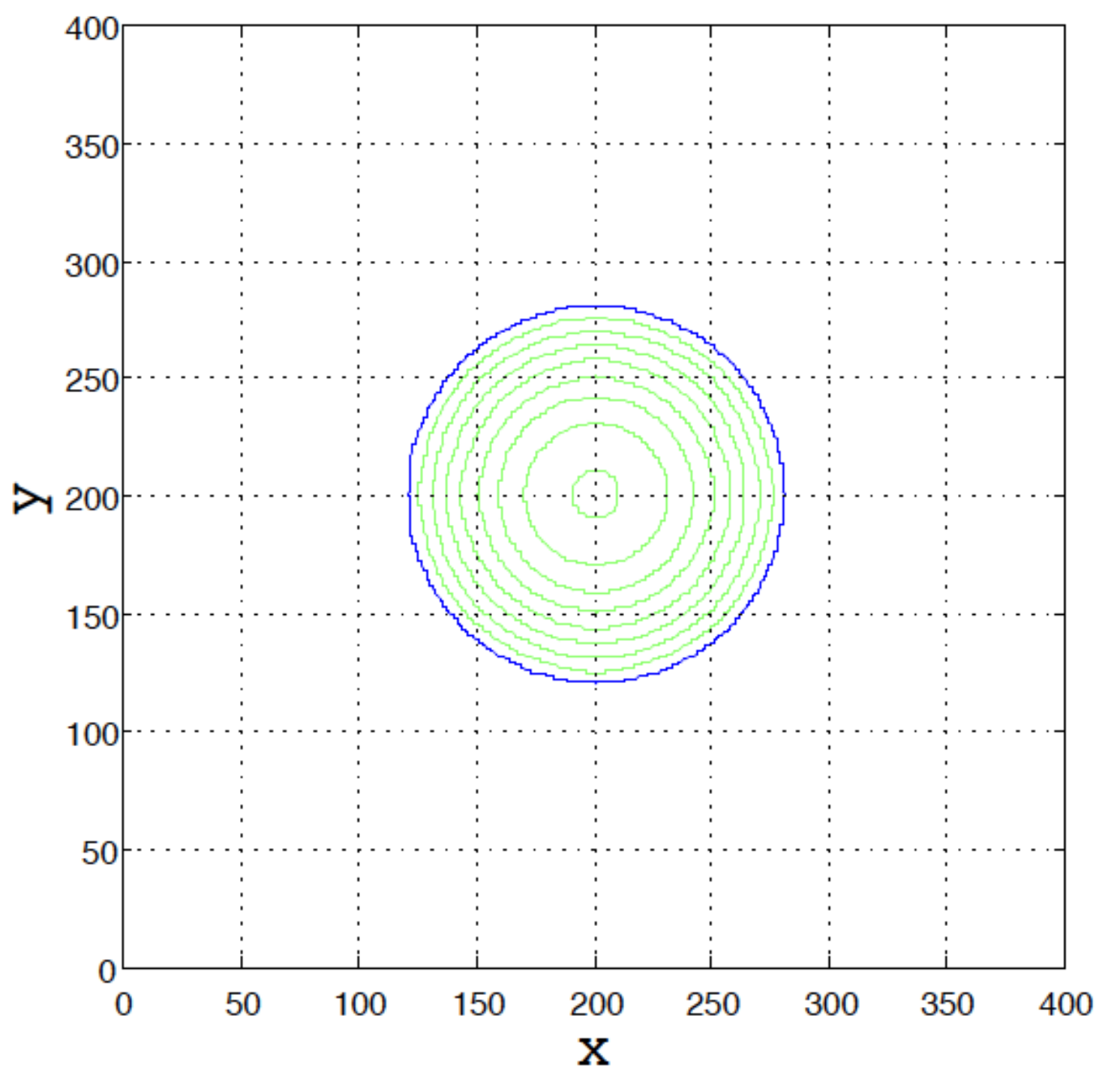}\\
Evolution of a circle.
\end{center}
\end{minipage}\hspace{50pt}
\begin{minipage}[c]{2.5in}
\begin{center}
%\epsfig{file=MMC-FIGURES/circle-1.eps, height=2.5in}\\
%\hspace{-60pt}
\includegraphics [width = 2.7in]{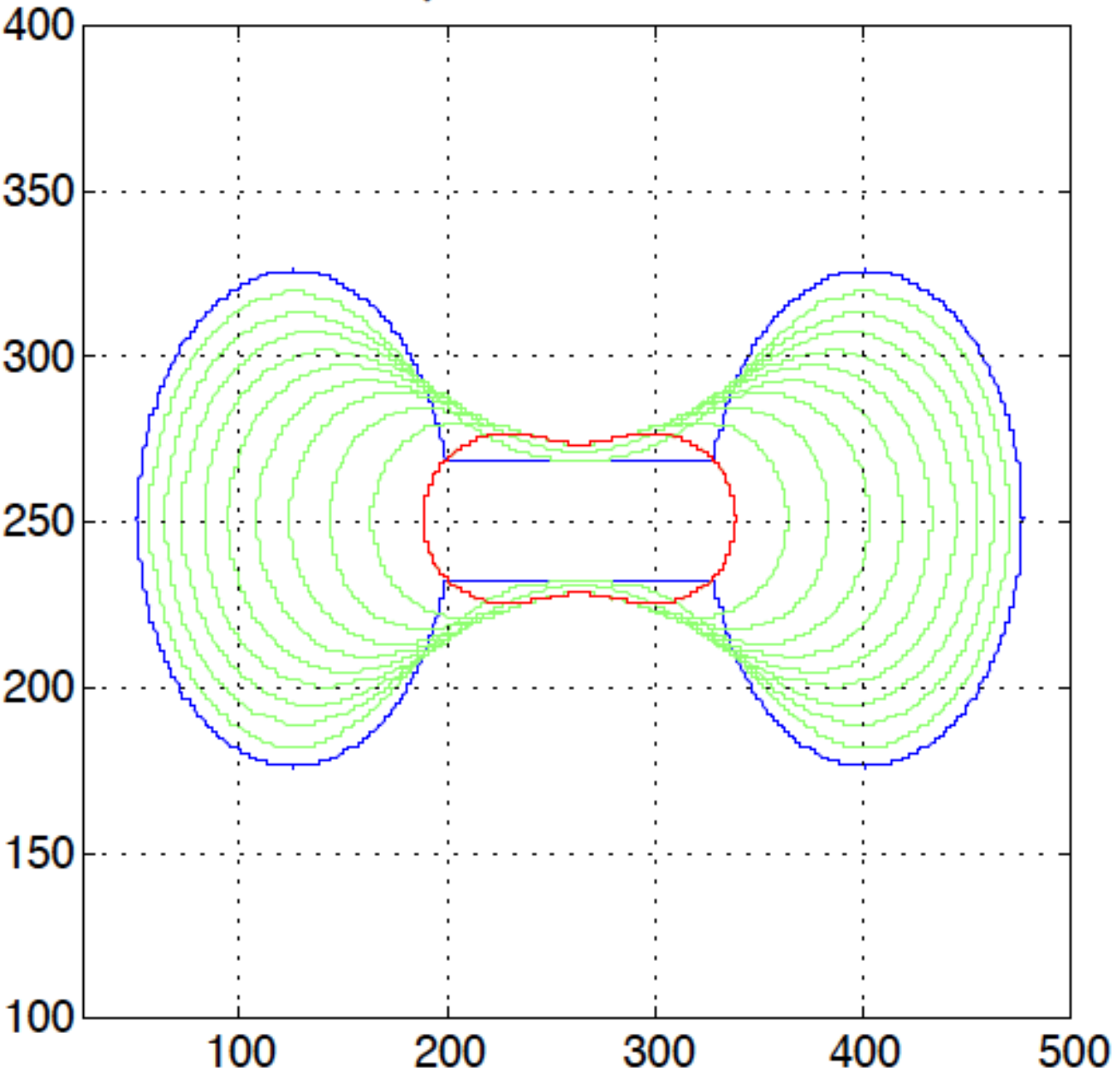}\\
Evolution of a dumpbell.
\end{center}
\end{minipage}

\vspace{10pt}

We investigate in detail the motion of a circle. The figures below show how the radius of a circle changes for several values of the parameter $\mu$ between $0.3$ and $1$, as well as for various values of $h$ between $0.05$ and $0.1$. The black circled curve is the exact mean curvature motion.

\begin{minipage}[c]{3in}
\begin{center}
\includegraphics [height = 2in]{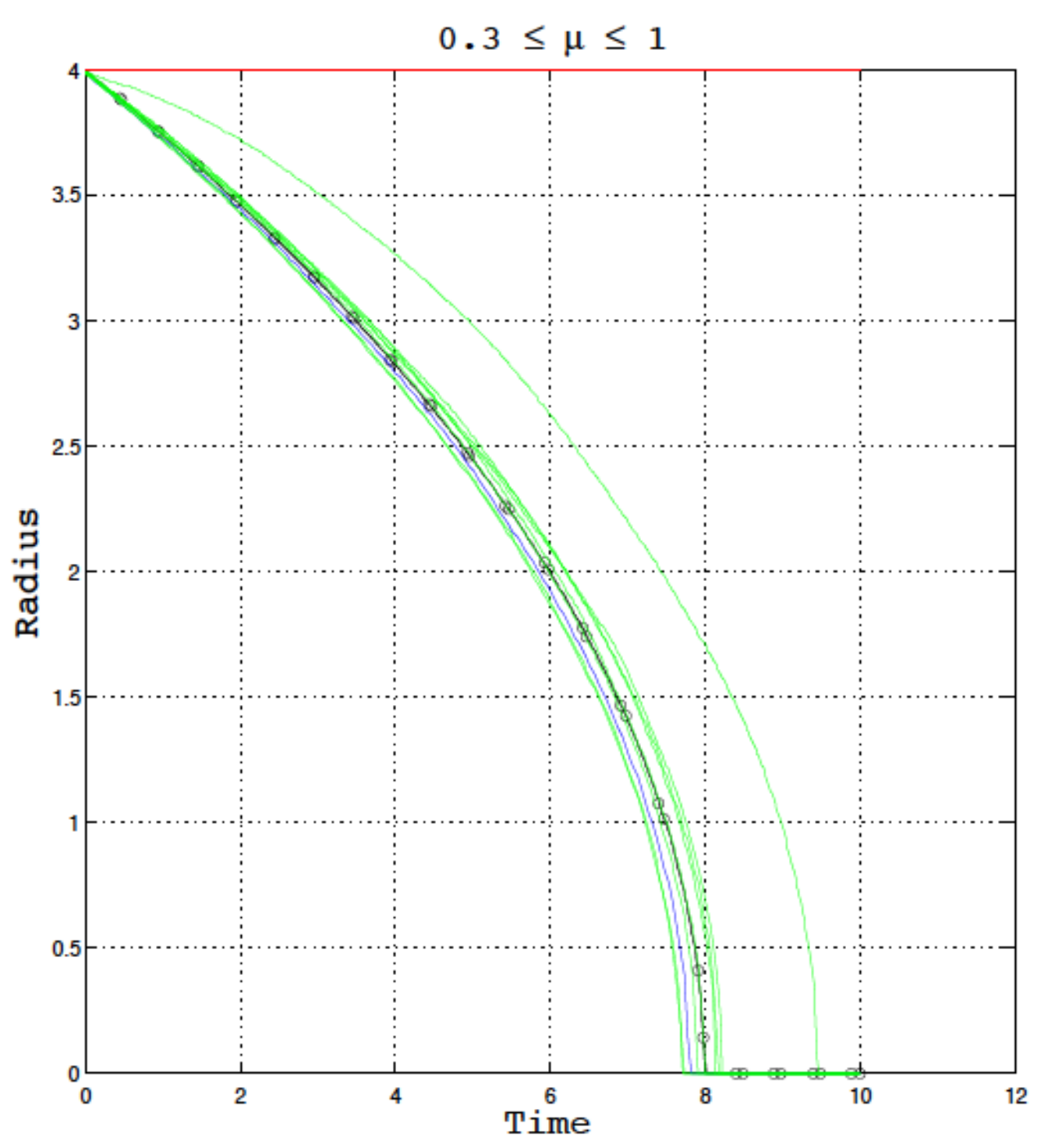}
\label{CirRadMeu}\\
Evolution of the circle radius with $\mu$.
\end{center}
\end{minipage}
\begin{minipage}[c]{3in}
\begin{center}
\includegraphics [height = 2in]{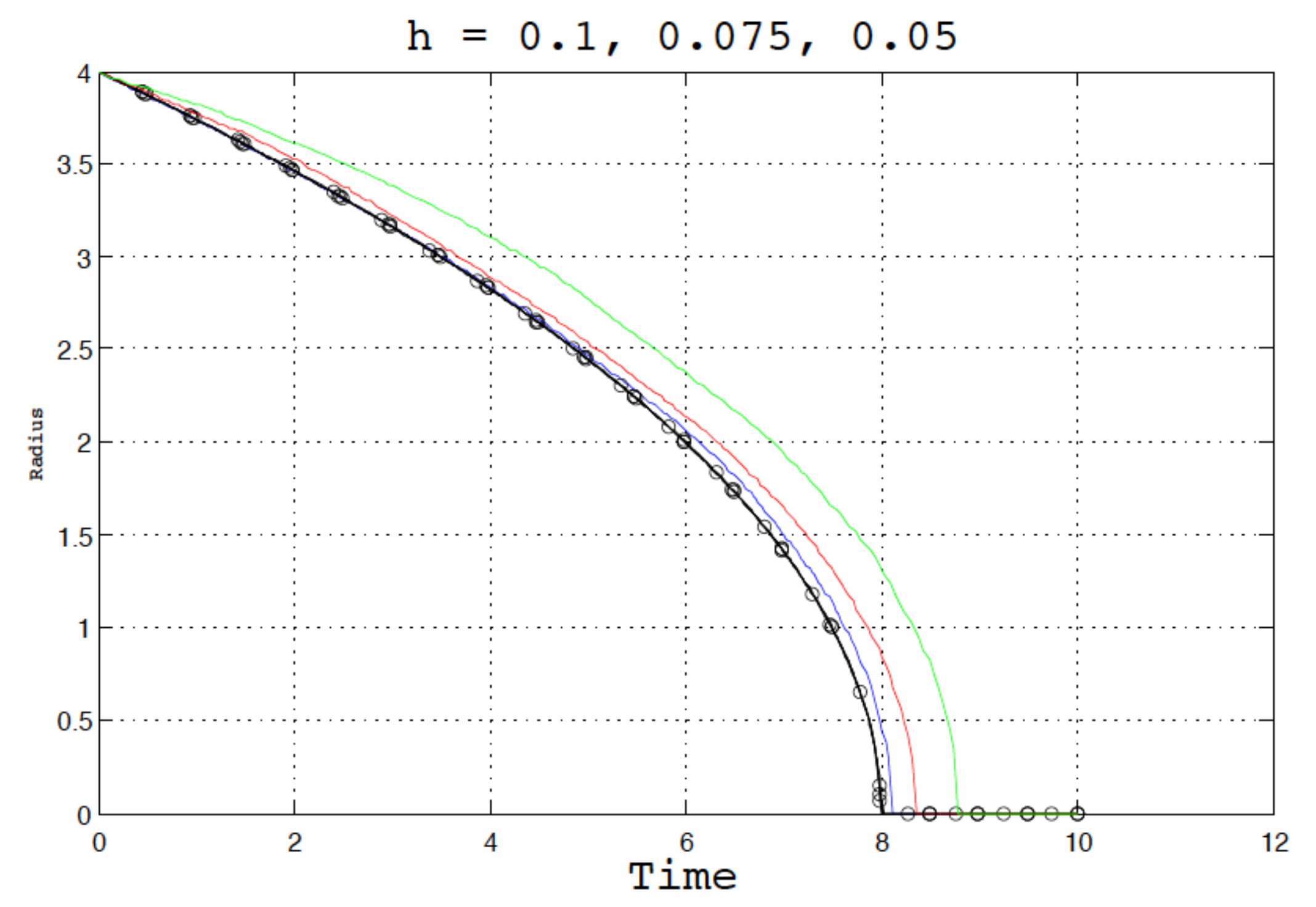}
\label{CirRadH}\\
Evolution of the circle radius with $h$.
\end{center}
\end{minipage}

\vspace{10pt}

The next figures illustrate the anisotropic behavior described in Theorem \ref{Th4}. We start with a finger-shaped domain, tilted at angle $45^{\circ}$ and verify that it moves with constant velocity.

\begin{minipage}[c]{3in}
\begin{center}
\includegraphics [height = 2in]{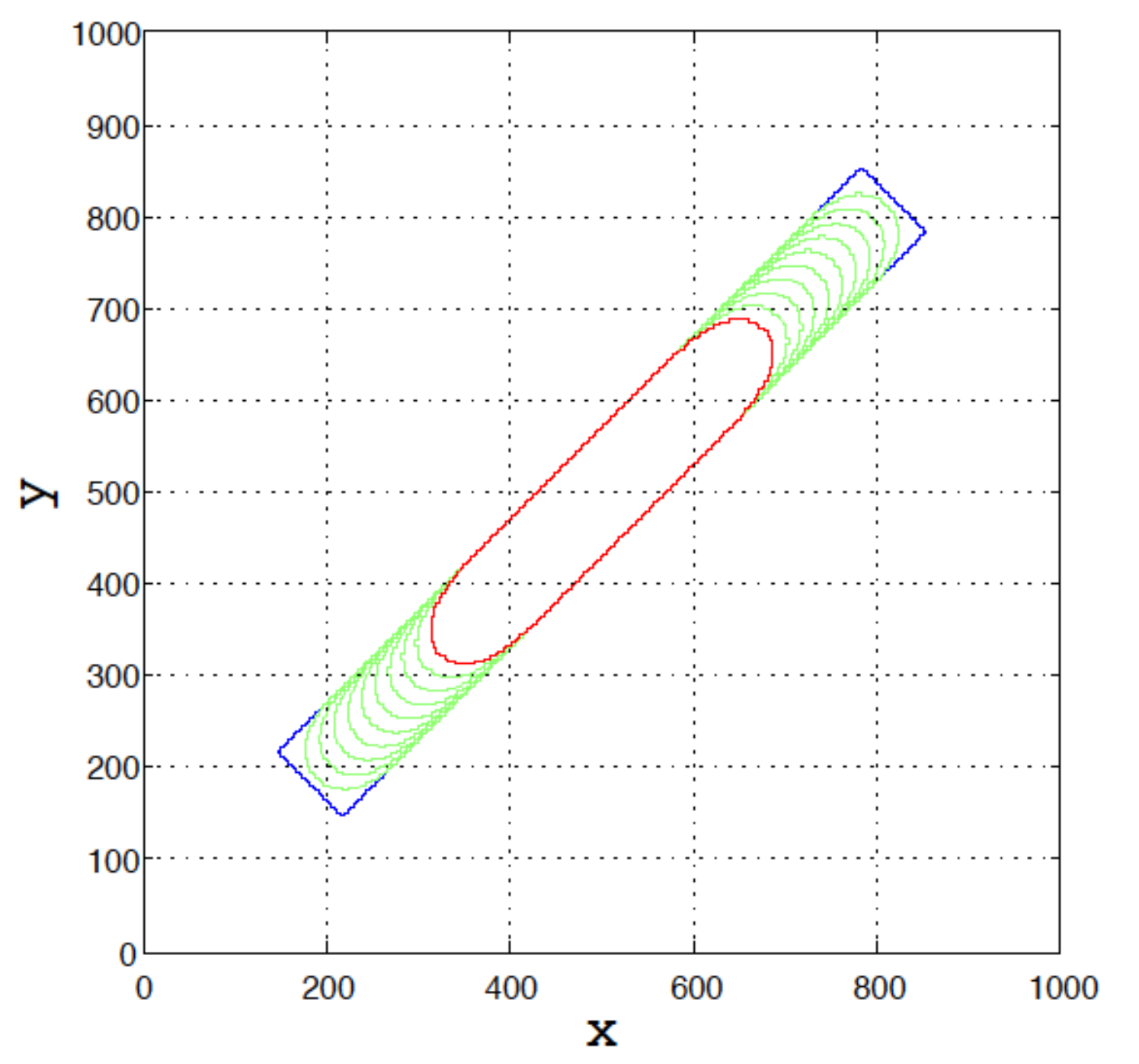}
\label{FigFin}\\
Angular motion example: $\theta = 45^{\circ}$
\end{center}
\end{minipage}
\begin{minipage}[c]{3in}
\begin{center}
\includegraphics [height = 2in]{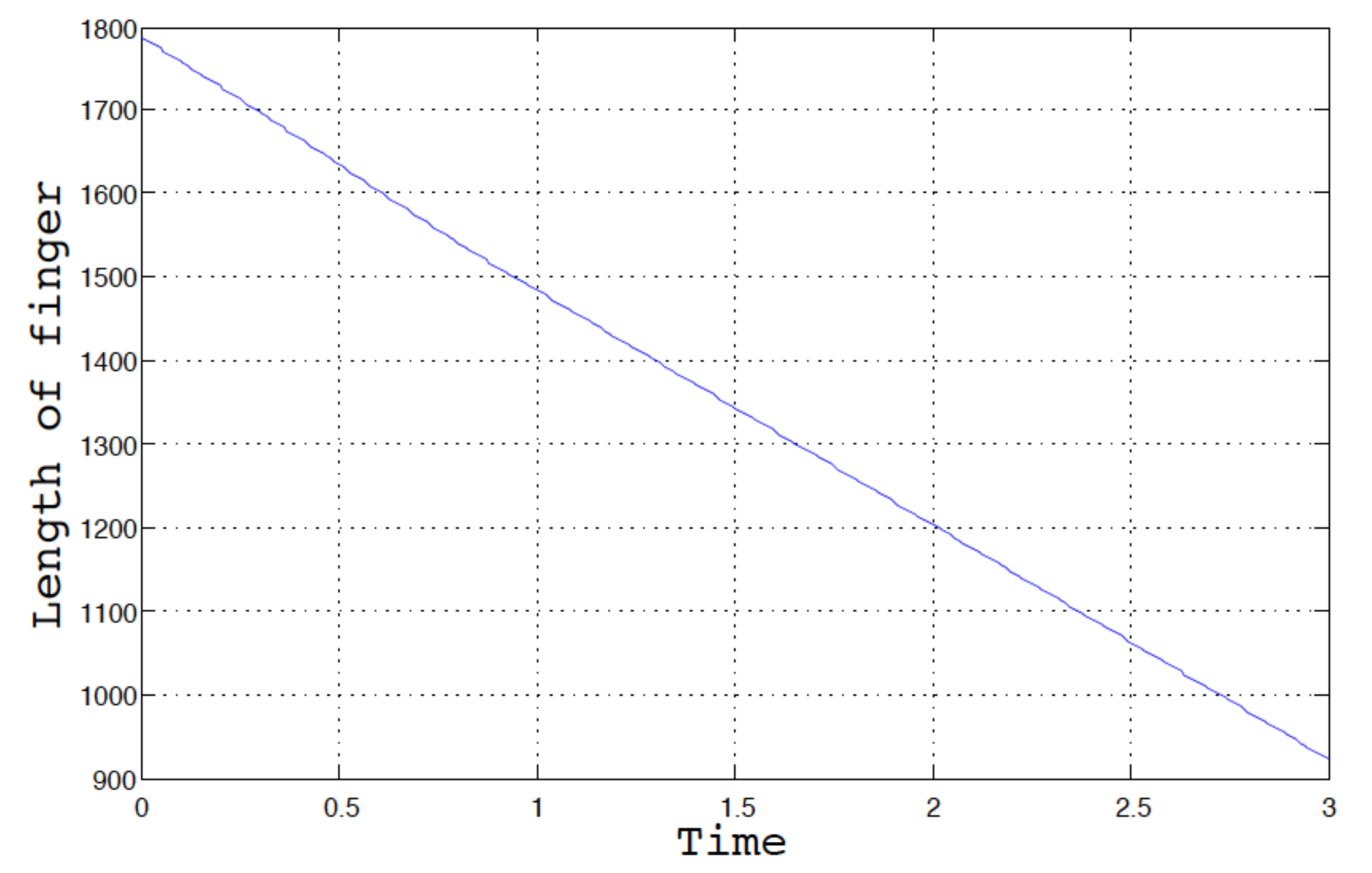}
\label{FigFinVel}\\
\vspace{-15pt}
Angular motion velocity: $\theta = 45^{\circ}$
\end{center}
\end{minipage}

\vspace{10pt}

We conclude the section with calculating the velocity of the finger-shaped domains for the angles $0^{\circ}, 1^{\circ}, 2^{\circ}, ..., 45^{\circ}$, with 5 values of $\mu$ between $0.5$ and $1$. The results are illustrated below.

\begin{center}
\includegraphics [width = 4in]{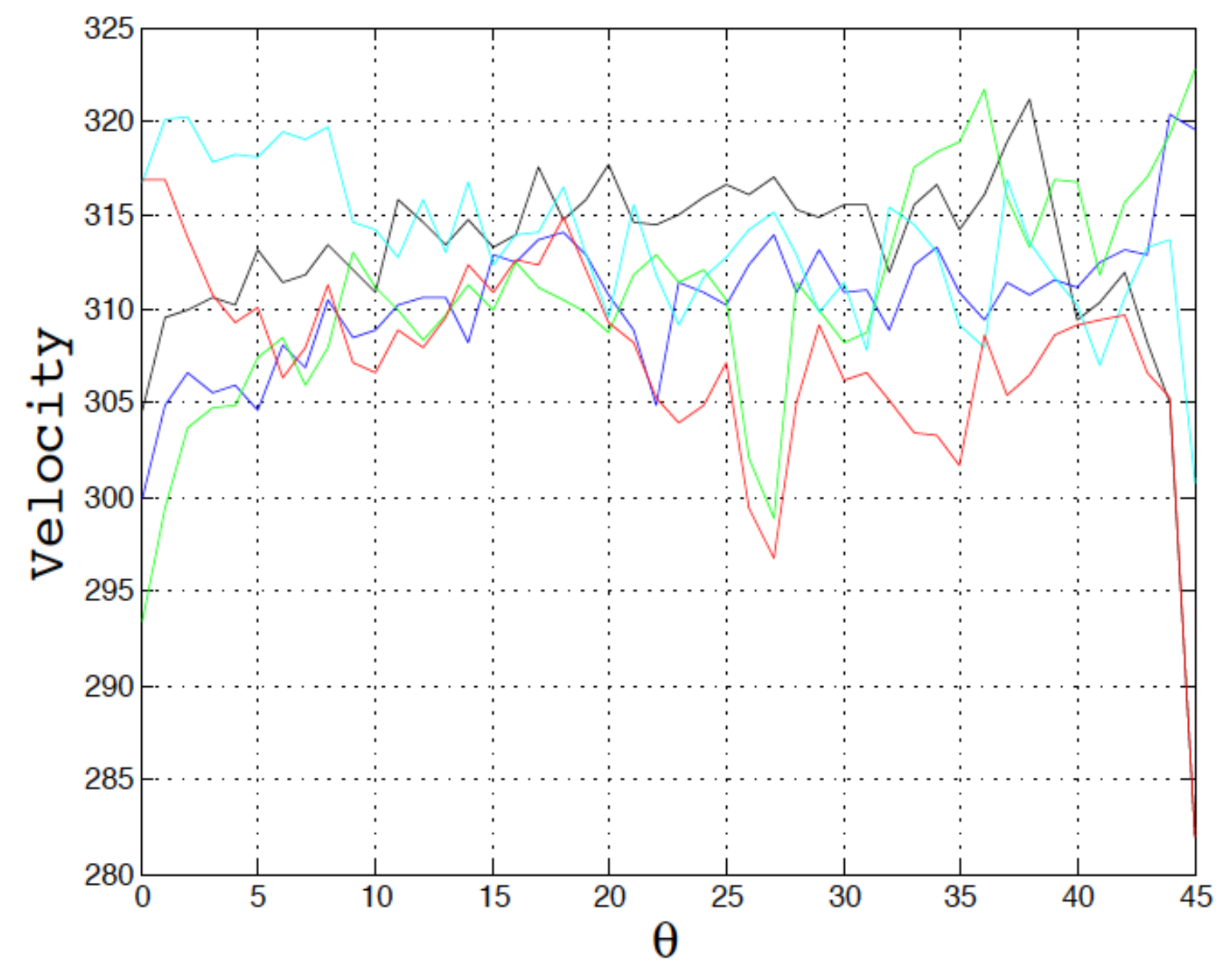}\\
%\label{VelAng}\\
Velocities for various angles and values of $\mu$.
\end{center}

%\begin{minipage}[b]{.49 \linewidth}
%\begin{center}
%\includegraphics [width = 2in, angle = 270]{MMC-FIGURES/circle-radius-h}\\
%Figure \circle-radius-h.
%\end{center}
%\end{minipage}

\section{Acknowledgement} The authors are grateful to Dr. Vadym Doroshenko for the help with the numerical simulation.

%\frac{2}{4 \pi} \frac{\sqrt{h}}{\sqrt{\mu}}\int_{0}^{\sqrt{\frac{2}{\mu}} \frac{\sqrt{n_0 q^2 - p l - m q}}{\sqrt{\kappa} (q^2 + p^2)^{\frac{3}{4}}}} e^{-\frac{(p^2 + q^2) x^2}{4}} dx + O(h \ln h) = \\

%\appendix
%\section{Heuristic argument for Allen-Cahn and Thresholding Scheme}
%\subsection{Allen-Cahn Equation}
%
%\subsection{Thresholding scheme}
%Radially symmetric case
%
%General case using Taylor expansion or Evans' approach
%
%\subsection{Convergence Scheme for Viscosity Solution}\label{ViscDefConv}

\bibliography{bibliographymmc.bib}

\begin{thebibliography}{10}

\bibitem{Abr}
M.~Abramowitz and I.~Stegun.
\newblock {\em {Handbook of Mathematical Functions with Formulas, Graphs and
  Mathematical Tables}}.
\newblock National Bureau of Standards, Applied Mathematics Series, 1964.

\bibitem{Barles}
G.~Barles and C.~Georgelin.
\newblock A simple proof of convergence for an approximation scheme for
  computing motions by mean curvature.
\newblock {\em SIAM J. Numer. Anal.}, 32(2):484--500, 1995.

\bibitem{SIAMControl}
G.~Barles, H.~Soner, and P.~Souganidis.
\newblock Front propagation and phase field theory.
\newblock {\em SIAM J. Control Optim.}, 31(2):439--469, 1993.

\bibitem{SouganidisBarlesARMA}
P.~Barles, G.;~Souganidis.
\newblock A new approach to front propagation problems: theory and
  applications.
\newblock {\em Arch. Rational Mech. Anal.}, 141(3):237 -- 296, 1998.

\bibitem{M}
B.~Bence, J.~Merriman, and S.~Osher.
\newblock Motion of multiple functions: a level set approach.
\newblock {\em J. Comput. Phys.}, 112(2):334--363, 1994.

\bibitem{KohnBronsard}
L.~Bronsard and R.~Kohn.
\newblock Motion by mean curvature as the singular limit of ginzburg-landau
  dynamics.
\newblock {\em J. Differential Equations}, 90(2):211--237, 1991.

\bibitem{ChenMaxPrin}
X.~Chen.
\newblock Generation and propagation of interfaces for reaction-diffusion
  equations.
\newblock {\em J. Differential Equations}, 96(1):116--141, 1992.

\bibitem{deMottoniSchatzman}
P.~de~Mottoni and M.~Schatzman.
\newblock Geometrical evolution of developed interfaces.
\newblock {\em Trans. Amer. Math. Soc.}, 347(5):1533--1589, 1995.

\bibitem{Evans}
L.~Evans.
\newblock Convergence of an algorithm for mean curvature motion.
\newblock {\em Indiana Univ. Math. J.}, 42(2):533--557, 1993.

\bibitem{ESS}
L.~Evans, H.~Soner, and P.~Souganidis.
\newblock Phase transitions and generalized motion by mean curvature.
\newblock {\em Comm. Pure Appl. Math.}, 45(9):1097--1123, 1992.

\bibitem{IlmanenAC}
T.~Ilmanen.
\newblock Convergence of the allen-cahn equation to brakke's motion by mean
  curvature.
\newblock {\em J. Differential Geom.}, 38(2):417--461, 1993.

\bibitem{Pires}
H.~Ishii, G.~Pires, and P.~Souganidis.
\newblock Threshold dynamics type approximation schemes for propagating fronts.
\newblock {\em J. Math. Soc. Japan}, 51(2):267--308, 1999.

\bibitem{watson}
G.~Watson.
\newblock {\em A treatise on the theory of Bessel functions}.
\newblock Cambridge University Press, 1922.

\end{thebibliography}

\end{document}